\documentclass[a4paper]{article}

\usepackage[T1]{fontenc}
\usepackage[latin9]{inputenc}
\usepackage{geometry}
\usepackage{amsmath}
\usepackage{amsthm}
\usepackage{amssymb}
\usepackage{bbm}
\usepackage{color}
\usepackage{paralist}
\usepackage[unicode=true,pdfusetitle,
 bookmarks=true,bookmarksnumbered=false,bookmarksopen=false,
 breaklinks=false,pdfborder={0 0 0},pdfborderstyle={},backref=false,colorlinks=false]
 {hyperref}

\makeatletter

\usepackage[nameinlink,capitalise,noabbrev]{cleveref}

\hypersetup{%
    bookmarksnumbered, bookmarksopen=true, bookmarksopenlevel=1,%
}

\theoremstyle{plain}
\newtheorem{thm}{Theorem}[section]
\crefname{thm}{Theorem}{Theorems}
\theoremstyle{plain}
\newtheorem{lem}[thm]{Lemma}
\crefname{lem}{Lemma}{Lemmas}
\theoremstyle{plain}
\newtheorem{cor}[thm]{Corollary}
\theoremstyle{plain}
\newtheorem*{claim*}{Claim}
\crefname{claim}{Claim}{Claims}
\theoremstyle{definition}
\newtheorem{defn}[thm]{Definition}
\theoremstyle{plain}
\newtheorem{conjecture}[thm]{Conjecture}
\crefname{conjecture}{Conjecture}{Conjectures}
\newtheorem{question}[thm]{Question}
\theoremstyle{plain}
\newtheorem{prop}[thm]{Proposition}
\crefname{prop}{Proposition}{Propositions}
\theoremstyle{definition}

\theoremstyle{definition}

\theoremstyle{plain}
\newtheorem{claim}[thm]{Claim}

\crefformat{equation}{#2(#1)#3}
\crefname{appsec}{Appendix}{Appendices}

\date{}

\crefformat{equation}{#2(#1)#3}

\let\originalleft\left
\let\originalright\right
\renewcommand{\left}{\mathopen{}\mathclose\bgroup\originalleft}
\renewcommand{\right}{\aftergroup\egroup\originalright}
\usepackage{verbatim}

\makeatletter
\renewcommand*{\UrlTildeSpecial}{%
  \do\~{%
    \mbox{%
      \fontfamily{ptm}\selectfont
      \textasciitilde
    }%
  }%  
}%    
\let\Url@force@Tilde\UrlTildeSpecial
\makeatother

\makeatother

\allowdisplaybreaks

\begin{document}
\title{Combinatorial anti-concentration inequalities, with applications}
\author{Jacob Fox\thanks{Department of Mathematics, Stanford University, Stanford, CA 94305.
Email: \href{mailto:jacobfox@stanford.edu} {\nolinkurl{jacobfox@stanford.edu}}.
Research supported by a Packard Fellowship and by NSF Career Award
DMS-1352121.}\and Matthew Kwan \thanks{Department of Mathematics, Stanford University, Stanford, CA 94305.
Email: \href{mailto:mattkwan@stanford.edu} {\nolinkurl{mattkwan@stanford.edu}}.
Research supported in part by SNSF project 178493.}\and Lisa Sauermann\thanks{Department of Mathematics, Stanford University, Stanford, CA 94305.
Email: \href{mailto:lsauerma@stanford.edu} {\nolinkurl{lsauerma@stanford.edu}}.}}

\maketitle
\global\long\def\RR{\mathbb{R}}%
\global\long\def\QQ{\mathbb{Q}}%
\global\long\def\E{\mathbb{E}}%
\global\long\def\Var{\operatorname{Var}}%
\global\long\def\CC{\mathbb{C}}%
\global\long\def\NN{\mathbb{N}}%
\global\long\def\ZZ{\mathbb{Z}}%
\global\long\def\GG{\mathbb{G}}%
\global\long\def\tallphantom{\vphantom{\sum}}%
\global\long\def\tallerphantom{\vphantom{\int}}%
\global\long\def\supp{\operatorname{supp}}%
\global\long\def\one{\mathbbm{1}}%
\global\long\def\d{\operatorname{d}}%
\global\long\def\Unif{\operatorname{Unif}}%
\global\long\def\Po{\operatorname{Po}}%
\global\long\def\Bin{\operatorname{Bin}}%
\global\long\def\Ber{\operatorname{Ber}}%
\global\long\def\Geom{\operatorname{Geom}}%
\global\long\def\Rad{\operatorname{Rad}}%
\global\long\def\floor#1{\left\lfloor #1\right\rfloor }%
\global\long\def\ceil#1{\left\lceil #1\right\rceil }%
\global\long\def\cond{\,\middle|\,}%
\global\long\def\g{\boldsymbol{\gamma}}%
\global\long\def\x{\boldsymbol{\xi}}%

\begin{abstract}
We prove several different anti-concentration inequalities for functions
of independent Bernoulli-distributed random variables. First, motivated
by a conjecture of Alon, Hefetz, Krivelevich and Tyomkyn, we prove
some ``Poisson-type'' anti-concentration theorems that give bounds
of the form $1/e+o\left(1\right)$ for the point probabilities of certain polynomials. Second, we
prove an anti-concentration inequality for polynomials with nonnegative
coefficients which extends the classical Erd\H os--Littlewood--Offord
theorem and improves a theorem of Meka, Nguyen and Vu for polynomials
of this type. As an application, we prove some new anti-concentration bounds for subgraph counts in random graphs.
\end{abstract}

\section{Introduction}

In probabilistic combinatorics (and probability in general), many
arguments are heavily dependent on \emph{concentration} inequalities,
which show that certain random variables are likely to lie in a small
interval around their mean. For example, if $X$ takes the binomial
distribution $\Bin\left(n,p\right)$, which has mean $\mu=np$ and
variance $\sigma^{2}=p\left(1-p\right)n$, then typically $X=\mu\pm O\left(\sigma\right)$.
\begin{comment},
and in fact $\Pr\left(\left|X-\mu\right|\ge t\right)$ decays at an
exponential rate as $t$ increases past $\sigma$.
 (in particular, we have $\Pr\left(\left|X-np\right|\ge t\right)\le e^{-\Omega\left(t^{2}/\left(\sigma^{2}+t\right)\right)}$).
See for example \cite{Led,BLM} for books on the concentration phenomenon.\end{comment}
In the other direction, \emph{anti-concentration} inequalities give
\emph{upper} bounds on the probability that a random variable falls
into a small interval or is equal to a particular value. The \emph{L\'evy
concentration function} $Q_{X}$ of a random variable $X$ is defined
by
\[
Q_{X}\left(t\right):=\sup_{x\in\RR}\Pr\left(x\le X\le x+t\right).
\]
Returning to the example $X\in\Bin\left(n,p\right)$, we can compute
$\Pr\left(X=x\right)=O\left(1/\sigma\right)$ for all $x\in\NN$,
which implies that $Q_{X}\left(t\right)=O\left(\left(t+1\right)/\sigma\right)$.
Bounds of this type can be proved for a variety of different kinds
of random variables. See for example \cite{NV13,Vu17} for
surveys on anti-concentration.

In the above example $X\in\Bin\left(n,p\right)$, in the case where
$p$ is fixed and $n$ is large, the above concentration and anti-concentration
phenomena can both be explained by comparison to a Gaussian distribution.
\begin{comment}
 Indeed, the central limit theorem shows that $X$ is approximately
normally distributed with mean $np$ and standard deviation $\sqrt{p\left(1-p\right)n}=\Theta\left(\sqrt{n}\right)$.
\end{comment}
 More generally, as an important example generalising the binomial
distribution, let $a_{1},\dots,a_{n}$ be a fixed sequence of nonzero
real numbers, let $\xi_{1},\dots,\xi_{n}$ be a sequence of i.i.d.\ $p$-Bernoulli-distributed random variables (meaning that $\Pr\left(\xi_{i}=1\right)=p$,
$\Pr\left(\xi_{i}=0\right)=1-p$) and let $X:=a_{1}\xi_{1}+\dots+a_{n}\xi_{n}$.
If $1\le |a_{i}|=O(1)$ for each $i$, then one can apply a quantitative central limit theorem to compare
$X$ to a Gaussian distribution and show that $\Pr\left(|X-x|<1\right)=O\left(1/\sqrt{n}\right)$ for
any $x\in\RR$ (and therefore $Q_{X}\left(t\right)=O\left(\left(t+1\right)/\sqrt{n}\right)$). Remarkably, the same result holds even when we require no upper bound on the $|a_{i}|$, meaning that $X$
may be far from Gaussian and may not even be particularly well-concentrated.
This is the content of the Erd\H os--Littlewood--Offord theorem\footnote{The Erd\H os--Littlewood--Offord theorem was most famously stated
for the case where $p=1/2$. This case is somewhat simpler because
we can assume that all the $a_{i}$ are positive: changing the sign
of $a_{i}$ only changes the distribution of $X$ by a translation.
However, with modern techniques it is not difficult to deduce a similar
estimate for any fixed $p\in\left(0,1\right)$; see for example \cite[Lemma~A.1]{BVW10}.}~\cite{Erd45}. A precursor to the Erd\H os--Littlewood--Offord
theorem was first used by Littlewood and Offord~\cite{LO43} in their
study of random polynomials more than 50 years ago, and since then,
the theorem and its variants have played an important role in probability,
especially in random matrix theory (see for example \cite{TV09a,TV09b}).

Observe that $a_{1}\xi_{1}+\dots+a_{n}\xi_{n}$ is a linear polynomial
in the $\xi_{i}$, so a natural variation on the Littlewood--Offord
problem is to consider polynomials of higher degree. This problem seems to have been first studied by Rosi\'{n}ski and Samorodnitsky~\cite{RS96} in connection with L\'evy chaos, but it was later popularised by Costello, Tao and Vu~\cite{CTV06}
when they used a quadratic variant of the Littlewood--Offord inequality in their proof of Weiss' conjecture that a random symmetric $\pm1$
matrix typically has full rank. Anti-concentration inequalities
for higher-degree polynomials have since found several applications
in the theory of Boolean functions (see for example \cite{MNV16,RV13}).
The current most general result is due to Meka, Nguyen, and Vu~\cite{MNV16},
%(proved using a theorem of Kane~\cite{Kan14}),
and gives a bound
in terms of the \emph{rank} of a polynomial, as follows. For a real multilinear
degree-$d$ polynomial $f$ in $n$ variables, consider the $d$-uniform
hypergraph on the vertex set $\left\{ 1,\dots,n\right\} $ with a
hyperedge $\left\{ i_{1},\dots,i_{d}\right\} $ if the coefficient
of $x_{i_{1}}\dots x_{i_{d}}$ in $f$ has absolute value at least 1. Then the rank of
$f$ is defined to be the largest matching in this hypergraph. For
example, if all $\binom{n}{d}$ degree-$d$ coefficients of $f$ have absolute value at least 1, then
$f$ has rank $\floor{n/d}=\Omega\left(n\right)$. Meka, Nguyen and
Vu proved that for fixed $d\in\NN$ and $p\in\left(0,1\right)$,
any multilinear degree-$d$ rank-$r$ polynomial $f$ in $n$ variables, any
$x\in\RR$, and $\x=\left(\xi_{1},\dots,\xi_{n}\right)\in\Ber\left(p\right)^{n}$, we have
\[
\Pr\left(|f\left(\x\right)-x|<1\right)\le\frac{\left(\log r\right)^{O\left(1\right)}}{\sqrt{r}}.
\]
Up to the polylogarithmic factor, this result is best-possible in multiple
regimes. Consider for example the polynomial $x_{1}\dots x_{d}+x_{d+1}\dots x_{2d}+\dots+x_{\left(r-1\right)d+1}\dots x_{rd}$, with only linearly many nonzero coefficients,
or the polynomial $(x_1+\dots+x_n)^d$ with $\Theta(n^d)$ nonzero coefficients.

Our first result is that if the coefficients of $f$ are nonnegative,
then we can remove the polylogarithmic factor in the Meka--Nguyen--Vu
theorem, even with a slightly looser notion of rank. For a multilinear
polynomial $f$ in $n$ variables, let $r\left(f\right)$ be the largest
matching in the (non-uniform) hypergraph on the vertex set $\left\{ 1,\dots,n\right\} $
with a hyperedge $\left\{ i_{1},\dots,i_{k}\right\} $ if the coefficient
of $x_{i_{1}}\dots x_{i_{k}}$ in $f$ has absolute value at least 1.
\begin{thm}
\label{thm:positive-polynomial-LO}Fix $d\in\NN$ and $p\in\left(0,1\right)$,
let $f$ be a degree-$d$ multilinear polynomial in $n$ variables
with nonnegative coefficients, and let $\x\in\Ber\left(p\right)^{n}$.
Then for any $x\in\RR$, with $r\left(f\right)$ as defined above,
we have
\[
\Pr\left(|f\left(\x\right)-x|<1\right)\le O\left(1/\sqrt{r\left(f\right)}\right).
\]
\end{thm}

We remark that polynomials of Bernoulli random variables with nonnegative
coefficients arise naturally in probabilistic combinatorics. An important example is the number of copies of a fixed graph $H$ in a random graph $\GG(n,p)$ (we will say more about this in \cref{subsec:subgraph-statistics}). Actually
there is also a rich theory of concentration inequalities for these
kinds of polynomials, due primarily to Kim and Vu (see \cite{KV00}
for a survey).

Actually, it seems that for many polynomials that arise in
combinatorics, their polynomial
structure is less important than the fact that they are strongly monotone:
changing some $\xi_{i}$ from 0 to 1 tends to cause a large increase
in the value of $f\left(\x\right)$. Our next result extends \cref{thm:positive-polynomial-LO}
in this setting.
\begin{thm}
\label{thm:rough-LO}Fix $p\in\left(0,1\right)$. Consider
a function $f:\left\{ 0,1\right\} ^{n}\to\RR$, let $\x\in\Ber\left(p\right)^{n}$,
and define the random variables
\[
\Delta_{i}\left(\x\right):=f\left(\xi_{1},\dots,\xi_{i-1},1,\xi_{i+1},\dots,\xi_{n}\right)-f\left(\xi_{1},\dots,\xi_{i-1},0,\xi_{i+1},\dots,\xi_{n}\right).
\]
Suppose for some positive $s$ (which may be a function of $n$) that
$\Pr\left(\Delta_{i}\left(\x\right)\le2s\right)\le n^{-\omega\left(1\right)}$ for all $i\in \{1,\dots,n\}$.
Then, for any $x\in\RR$,
\[
\Pr\left(\left|f\left(\x\right)-x\right|<s\right)\le \max_{t}\binom{n}{t}p^{t}\left(1-p\right)^{n-t}+o\left(1/\sqrt n\right)=O\left(1/\sqrt n\right).
\]
\end{thm}

We will prove \cref{thm:positive-polynomial-LO,thm:rough-LO} in \cref{sec:positive-polynomial-LO}.
Both proofs are quite similar (and quite short), and proceed along
similar lines to Erd\H os' original proof of the Erd\H os--Littlewood--Offord
theorem: the events in question are ``almost''
antichains in the $n$-dimensional Boolean lattice. We remark that the main term of \cref{thm:rough-LO} is best-possible: consider the case $f(\x)=\xi_1+\dots +\xi_n$, with any $s<1/2$.

The above discussion concerns the regime where $p$ is fixed
and $n$ is large, in which case we expect anti-concentration behaviour
to be ``Gaussian-like''. However, if $p$ is allowed to be a decaying
function of $n$, then we cannot hope for bounds as strong as $O\left(1/\sqrt{n}\right)$.
Indeed, consider the case $X\in\Bin\left(n,p\right)$ with $p=1/n$.
The Poisson limit theorem (see for example \cite[p.~64]{Kal}) shows
that $X$ is asymptotically Poisson, which implies that $\Pr\left(X=x\right)\le1/e+o\left(1\right)$
for each $x\in\NN$. To our knowledge there is not yet a theory of
anti-concentration that generalises this fact, though a recent conjecture of
Alon, Hefetz, Krivelevich and Tyomkyn~\cite{AHKT} hints at the existence
and utility of such a theory. We discuss this in the next subsection, and in \cref{subsec:poisson} we will present some general ``Poisson-type'' inequalities for certain polynomials.

\subsection{Edge-statistics in graphs\label{subsec:edge-statistics}}

For an $n$-vertex graph $G$ and some $0\le k\le n$, consider a
uniformly random set of $k$ vertices $A\subseteq V\left(G\right)$
and define the random variable $X_{G,k}:=e\left(G\left[A\right]\right)$
to be the number of edges induced by the random $k$-set $A$. Motivated
by connections to graph inducibility\footnote{Roughly speaking, the \emph{inducibility} of a graph $H$ measures
the maximum number of induced copies of $H$ a large graph can have.
This notion was introduced in 1975 by Pippenger and Golumbic~\cite{PG75},
and has enjoyed a recent surge of interest; see for example \cite{BHLP16,HT,Yus,KNV}.}, Alon, Hefetz, Krivelevich and Tyomkyn~\cite{AHKT} recently initiated
the study of the anti-concentration of $X_{G,k}$, and made the following three conjectures.
\begin{conjecture}[{\cite[Conjecture~6.2]{AHKT}}]
\label{conj:sqrt}Suppose $k\to\infty$ and $n/k\to\infty$, and consider
$\ell$ satisfying $\ell=\Omega\left(k^{2}\right)$ and $\binom{k}{2}-\ell=\Omega\left(k^{2}\right)$.
Then $\Pr\left(X_{G,k}=\ell\right)=O\left(1/\sqrt{k}\right)$.
\end{conjecture}

\begin{conjecture}[{\cite[Conjecture~6.1]{AHKT}}]
\label{conj:omega}Suppose $k\to\infty$ and $n/k\to\infty$, and
consider $\ell$ satisfying $\ell=\omega\left(k\right)$ and $\binom{k}{2}-\ell=\omega\left(k\right)$.
Then $\Pr\left(X_{G,k}=\ell\right)=o\left(1\right)$.
\end{conjecture}

\begin{conjecture}[{\cite[Conjecture~1.1]{AHKT}}]
\label{conj:1/e}Suppose $k\to\infty$ and $n$ grows sufficiently rapidly in terms of $k$. Then
for all $0<\ell<\binom{k}{2}$ we have $\Pr\left(X_{G,k}=\ell\right)\le1/e+o\left(1\right)$.
\end{conjecture}

There has already been a lot of progress on these conjectures. Kwan,
Sudakov and Tran~\cite{KST} proved \cref{conj:omega} and proved that
in the setting of \cref{conj:sqrt}, we have $\Pr\left(X_{G,k}=\ell\right)=\left(\log k\right)^{O\left(1\right)}/\sqrt{k}$.
Combining the results of \cite{KST} with several new ideas, \cref{conj:1/e}
was then proved, independently by Fox and Sauermann~\cite{FS} and
by Martinsson, Mousset, Noever and Truji\'c~\cite{MMNT}.

Actually, Kwan, Sudakov and Tran's work on \cref{conj:omega,conj:sqrt}
involved an application of the Meka--Nguyen--Vu polynomial anti-concentration
inequality mentioned earlier. To illustrate the connection between
polynomial anti-concentration and this problem, instead of the random
size-$k$ subset $A\subseteq V\left(G\right)$, consider the closely
related random subset $A^{\Ber}\subseteq V\left(G\right)$, where
each of the $n$ vertices is included with probability $k/n$ independently.
Then, $X_{G,k}^{\Ber}:=e\left(G\left[A^{\Ber}\right]\right)$ can be
interpreted as a quadratic polynomial of a $\Ber\left(k/n\right)^{n}$-distributed
random vector, whose coefficients correspond to edges in the graph.

As our first application of our new anti-concentration theorems, we observe that \cref{thm:rough-LO} can be used to prove a stronger ``Bernoulli version'' of \cref{conj:sqrt},
in the more general setting of hypergraphs.
\begin{prop}
\label{prop:bernoulli-sqrt}Fix $r\in\NN$, suppose $k\to\infty$
and $n\ge2k$, and consider any $\ell$ satisfying $\ell=\Omega\left(k^{r}\right)$.
Then for any $r$-uniform hypergraph $G$, we have $\Pr\left(\left|X_{G,k}^{\Ber}-\ell\right|\le k^{r-1}\right)=O\left(1/\sqrt{k}\right)$.
\end{prop}

Note that \cref{prop:bernoulli-sqrt} is best-possible,
as can be seen by considering the case where $G$ is a clique. We
defer the proof of \cref{prop:bernoulli-sqrt} to \cref{sec:edge-ber}.

While the Littlewood--Offord point of view has been very useful for attacking
\cref{conj:sqrt}, the proofs of \cref{conj:1/e} in \cite{FS,MMNT}
proceeded along rather different lines. Our original motivation for
developing Poisson-type analogues to the Littlewood--Offord problem (where $p$ may go to zero) was that they may give a simpler proof of \cref{conj:1/e}
and faciliate generalisation to hypergraphs (a problem that was also
suggested by Alon, Hefetz, Krivelevich and Tyomkyn). While we did not manage to achieve this original goal, we were able to prove several
Poisson-type anti-concentration inequalities (stated in the next subsection),
one of which (\cref{thm:poisson-nonnegative}) implies the following
``Bernoulli version'' of \cref{conj:1/e}. The short deduction can
be found in \cref{sec:edge-ber}. 
\begin{prop}
\label{prop:bernoulli-1-e}Suppose $n/k\to\infty$. Then for any $\ell\ne0$
and any $r$-uniform hypergraph $G$, we have $\Pr\left(X_{G,k}^{\Ber}=\ell\right)\le1/e+o\left(1\right)$.
\end{prop}

\subsection{Poisson-type anti-concentration inequalities for polynomials\label{subsec:poisson}}

Consider first the Littlewood--Offord case where $X=a_{1}\xi_{1}+\dots+a_{n}\xi_{n}$,
for some fixed sequence $\left(a_{1},\dots,a_{n}\right)\in\RR^{n}$
and a random vector $\left(\xi_{1},\dots,\xi_{n}\right)\in\Ber\left(p\right)^{n}$.
We want to prove an anti-concentration theorem for the case where $p$
is small. Of course, if $p$ is \emph{extremely} small, then we are
very likely to see $\xi_{1}=\dots=\xi_{n}=0$, meaning that $\Pr\left(X=0\right)\approx1$.
Discounting this trivial case, we are able to prove the following
theorem (in \cref{sec:poisson}).
\begin{thm}
\label{thm:poisson-L-O}Consider a sequence $\left(a_{1},\dots,a_{n}\right)\in\RR^{n}$,
let $\x\in\Ber\left(p\right)^{n}$
and let $X:=a_{1}\xi_{1}+\dots+a_{n}\xi_{n}$. Then for any $x\ne0$,
\[
\Pr\left(X=x\right)\le\frac{1}{e}+o_{p\to0}\left(1\right).
\]
\end{thm}

(The notation $o_{p\to0}\left(1\right)$ refers to a function $g\left(p\right)$,
not depending on $n$, such that $g\left(p\right)\to0$ as $p\to0$). We remind the reader that in the case where $p$ does not tend to zero, the Littlewood--Offord theorem gives a bound of $O(1/\sqrt n)$ on the point probabilities of $X$.

One might hope to prove that the same conclusion
holds whenever $X$ is a polynomial of bounded degree with zero constant
coefficient. Unfortunately, this is not true in general: for example,
if $X=\sum_{i=1}^{n}\xi_{i}-\sum_{i=1}^{n}\sum_{j=i+1}^{n}\xi_{i}\xi_{j}$,
and $p=1/n$, then $\Pr\left(X=1\right)=3/\left(2e\right)+o\left(1\right)$.
Nevertheless, we are able to prove that $\Pr\left(X=x\right)$ is bounded away from 1 for any $x\ne0$, as follows.
\begin{prop}
\label{prop:weak-bound}Consider an $n$-variable polynomial $f$
with degree at most $d$, and let $\x\in\Ber\left(p\right)^{n}$
for some $p\le1/2$. Then for any $x$ not equal to the constant coefficient
of $f$,
\[
\Pr\left(f\left(\x\right)=x\right)\le1-2^{-d}.
\]
\end{prop}
We prove \cref{prop:weak-bound} in \cref{sec:poisson}, with the combinatorial Nullstellensatz (see \cite{Alo99}). Next, one way to recover the ``$1/e$'' behaviour is to consider only
polynomials with nonnegative coefficients, as in \cref{thm:positive-polynomial-LO}.
We also prove the following theorem in \cref{sec:poisson}.
\begin{thm}
\label{thm:poisson-nonnegative}Consider an $n$-variable polynomial
$f$ with nonnegative coefficients, and consider a random vector $\x\in\Ber\left(p\right)^{n}$.
Then for any $x$ not equal to the constant coefficient of $f$,
\[
\Pr\left(f\left(\x\right)=x\right)\le\frac{1}{e}+o_{p\to0}\left(1\right).
\]
\end{thm}

We emphasise that \cref{thm:poisson-nonnegative} makes no assumption
on the degree of the polynomial $f$.

\subsection{Subgraph counts in random graphs\label{subsec:subgraph-statistics}}

Fix $p\in\left(0,1\right)$ and let $G\in\GG\left(n,p\right)$ be
a random labelled graph on the vertex set $\left\{ 1,\dots,n\right\} $
where every pair of vertices is included as an edge with probability
$p$ independently. This is called the binomial or Erd\H os--Renyi
model of random graphs. For a fixed graph $H$, let $X_{H}$ be the
number of copies of $H$ in $G$. The study of $X_{H}$ and its distribution
is a fundamental topic in the theory of random graphs (see for example
\cite{Bol01,JLR00}). It is well-known that for any $H$ with no isolated
vertices, $X_{H}$ satisfies a central limit theorem, but the anti-concentration
behaviour of $X_{H}$ is not as well-understood. In this setting where
$p$ is fixed, one can deduce\footnote{The central limit theorem of Barbour, Karo\'nski and Ruci\'nski
is not stated with a metric that allows one to directly read off an
estimate for the distribution function of $X_{H}$. But, it is possible
to deduce such an estimate with the method of \cite[Proposition~1.2.2]{Ros11}.} from a quantitative central limit theorem by Barbour, Karo\'nski
and Ruci\'nski~\cite{BKR89} that $\Pr\left(X_{H}=x\right)\le O\left(1/\sqrt{n}\right)$
for all $x$. As an application of their Littlewood--Offord-type
polynomial anti-concentration inequality mentioned earlier in this
paper, Meka, Nguyen and Vu proved the stronger bound that $\Pr\left(X_{H}=x\right)\le n^{o(1)-1}$. This was a consequence of a more general result concerning random graphs of the form $G_p$, obtained by starting with a fixed graph $G$ and including each edge of $G$ with probability $p$ independently. Specifically, Meka, Nguyen and Vu observed that if $G$ has $r$ edge-disjoint copies of $H$, and $X_H(G_p)$ is the number of copies of $H$ in $G_p$, then $X_H(G_p)$ can be interpreted as a rank-$r$ polynomial of independent $p$-Bernoulli random variables, so $\Pr(X_H(G_p)=x)\le r^{o(1)-1/2}$ for all $x\in \NN$. Since the polynomial corresponding to $X_H(G_p)$ has nonnegative coefficients, one can use \cref{thm:positive-polynomial-LO} in place of the Meka--Nguyen--Vu
anti-concentration inequality to improve this as follows.
\begin{cor}
Fix $p\in (0,1)$ and let $G$ be a graph with $r$ edge-disjoint copies of $H$. Then for any $x\in \NN$ we have
$$\Pr\left(X_{H}(G_p)=x\right)=O\left(1/\sqrt{r}\right).$$
In particular, in $\GG(n,p)$ we have
$$\Pr\left(X_{H}=x\right)\le O(1/n).$$
\end{cor}
We believe that in $\GG(n,p)$, the above bound is far from optimal.
\begin{conjecture}
\label{conj:H-count}Fix $p\in\left(0,1\right)$ and fix a graph $H$
with $h$ non-isolated vertices. Let $G\in\GG\left(n,p\right)$. Then
for any $x\in\NN$,
\[
\Pr\left(X_{H}=x\right)=O\left(1/\sqrt{\Var\left(X_{H}\right)}\right)=O\left(1/n^{h-1}\right).
\]
\end{conjecture}

\cref{conj:H-count} would imply that $Q_{X_{H}}\left(t\right)=O\left(\left(t+1\right)n^{1-h}\right)$.
If true, this is best-possible; anything stronger would contradict
the central limit theorem known to hold for $X_{H}$. Although it
is not obvious how to prove \cref{conj:H-count}, we can use \cref{thm:rough-LO}
to obtain the optimal bound $Q_{X_{H}}\left(n^{h-2}\right)=O\left(1/n\right)$
for anti-concentration at a ``coarse'' scale.
\begin{thm}
\label{thm:general-random-graphs}Fix $p\in\left(0,1\right)$ and
fix a graph $H$ with $h$ vertices and at least one edge. Let $G\in\GG\left(n,p\right)$.
Then for any $x\in\NN$,
\[
\Pr\left(\left|X_{H}-x\right|\le n^{h-2}\right)=O\left(1/n\right).
\]
\end{thm}

The short deduction of \cref{thm:general-random-graphs} appears in
\cref{sec:subgraph-counts}.

With a bit more effort, one can combine \cref{thm:rough-LO} with some inductive arguments to prove an almost-optimal bound in the case where $H$ is a clique.

\begin{thm}
\label{thm:clique-counts}Fix $p\in\left(0,1\right)$ and $h\in\NN$.
Then $\Pr\left(X_{K_{h}}=x\right)\le n^{o\left(1\right)+1-h}$ for
all $x\in\NN$.
\end{thm}

The proof of \cref{thm:clique-counts} appears in \cref{sec:subgraph-counts}. We note that, after we had proved \cref{thm:clique-counts} and were working on writing this paper, Berkowitz~\cite{Ber18} released a preprint proving a \emph{local limit theorem} that gives an asymptotic estimate for the point probabilities of $X_{K_h}$ in terms of the density of a normal distribution (see also~\cite{GK16,Ber16}). This local limit theorem directly implies \cref{thm:clique-counts} and a strengthening of \cref{conj:H-count} in the case where $H$ is a clique. However, we still feel that it is worthwhile to include the proof of \cref{thm:clique-counts} in this paper: our proof is simpler and more combinatorial, and with some more work the ideas can be generalised to give a comparable bound for a larger class of subgraphs $H$. In a separate paper~\cite{FKS2} we will introduce some additional ideas to generalise \cref{thm:clique-counts} to all connected $H$.

We remark that the number of cliques of each size is determined by
the Tutte polynomial of a graph (see for example \cite[Theorem~2.4]{MN04}),
so \cref{thm:clique-counts} has the following corollary.
\begin{cor}
\label{cor:tutte}
The probability that two independently chosen random graphs from $\GG\left(n,1/2\right)$
have the same Tutte polynomial is $n^{-\omega\left(1\right)}$.
\end{cor}

\cref{cor:tutte} improves on a bound of $O\left(1/\log n\right)$ for this probability due to Loebl,
Matou\v sek and Pangr\'ac~\cite[Corollary~1.3]{LMP04} (the study
of this question was motivated by a conjecture of Bollob\'as, Pebody
and Riordan~\cite{BPR00} that almost all graphs are determined by
their Tutte polynomial).

\subsection{Structure of the paper and outline of the proofs}

The rest of the paper is organised as follows. First, in \cref{sec:positive-polynomial-LO}
we prove \cref{thm:positive-polynomial-LO,thm:rough-LO}. The rough idea for both proofs is the same, and is motivated by Erd\H os' proof of the Littlewood--Offord theorem. We consider a process that flips the bits $\xi_i$ from zero to one in a random order, where we start with $\x$ being the all-zero vector, and end with $\x$ being the all-one vector. We show that our random variable $f(\x)$ tends to increase fairly substantially on each flip (for \cref{thm:positive-polynomial-LO}, this is where we use the assumption that the coefficients are nonnegative). We deduce that during our process, $f(\x)$ does not tend to spend very long in the vicinity of any given value. This can then be translated into an anti-concentration result.

Next, in \cref{sec:poisson} we prove \cref{thm:poisson-L-O}, \cref{prop:weak-bound} and \cref{thm:poisson-nonnegative}. First, \cref{prop:weak-bound} has a fairly routine proof, using the combinatorial Nullstellensatz. Second, \cref{thm:poisson-L-O,thm:poisson-nonnegative} are proved in a unified way, via a careful induction on $n$.

%Next, in \cref{sec:poisson} we prove \cref{thm:poisson-L-O}, \cref{prop:weak-bound} and \cref{thm:poisson-nonnegative}. First, \cref{prop:weak-bound} has a fairly routine proof, using the combinatorial Nullstellensatz. Second, \cref{thm:poisson-L-O,thm:poisson-nonnegative} are proved in a unified way. The different assumptions (linearity and nonnegativity) in the statements of \cref{thm:poisson-L-O,thm:poisson-nonnegative} are two different ways to ensure that in the special case where all linear coefficients are exactly equal to $x$, the only way to have $f(\x)=x$ is when exactly one of the $\xi_i$ is equal to one. So, we only need to consider the case where one of the linear coefficients is not equal to $x$, which can be handled with a careful induction on $n$.

In the next two sections we give some applications: in \cref{sec:edge-ber}
we prove \cref{prop:bernoulli-sqrt,prop:bernoulli-1-e},
and in \cref{sec:subgraph-counts} we prove \cref{thm:general-random-graphs,thm:clique-counts}. These are all quite direct deductions from the theorems proved in \cref{sec:positive-polynomial-LO,sec:poisson}, with the exception of \cref{thm:clique-counts}. Roughly speaking, the idea for the proof of \cref{thm:clique-counts} is to fix a vertex $v$ and then decompose the random variable $X_{K_h}(G)$ (counting copies of $K_h$ in $G\in\GG(n,p)$) as $X_{K_{h}}=X_{K_{h}}\left(G-v\right)+X_{K_{h-1}}(G\left[N_{G}\left(v\right)\right])$. That is to say, every copy of $K_h$ in $G$ either does not use the vertex $v$, or it is comprised of the vertex $v$ and a copy of $K_{h-1}$ inside the neighbourhood $N_G(v)$ of $v$. We then apply \cref{thm:rough-LO} to $X_{K_{h}}\left(G-v\right)$, and deal with $X_{K_{h-1}}(G\left[N_{G}\left(v\right)\right])$ by induction on $h$. The main challenge for this approach is that the random variables $X_{K_{h}}\left(G-v\right)$ and $X_{K_{h-1}}(G\left[N_{G}\left(v\right)\right])$ are not independent.

Finally, \cref{sec:concluding} contains some concluding remarks,
including some open questions and some further miscellaneous results.

\subsection{Notation}

We use standard asymptotic notation throughout, and all asymptotics
are as $n\to\infty$ unless stated otherwise. For functions $f=f\left(n\right)$
and $g=g\left(n\right)$ we write $f=O\left(g\right)$ to mean there
is a constant $C$ such that $\left|f\right|\le C\left|g\right|$,
we write $f=\Omega\left(g\right)$ to mean there is a constant $c>0$
such that $f\ge c\left|g\right|$ for sufficiently large $n$, we
write $f=\Theta\left(g\right)$ to mean that $f=O\left(g\right)$
and $f=\Omega\left(g\right)$, and we write $f=o\left(g\right)$ or
$g=\omega\left(f\right)$ to mean that $f/g\to0$ as $n\to\infty$.

We also use standard graph-theoretic notation: $V\left(G\right)$ and
$E\left(G\right)$ are the sets of vertices and (hyper)edges of a
(hyper)graph $G$, and $v\left(G\right)$ and $e\left(G\right)$ are
the cardinalities of these sets. The subgraph of $G$ induced by a
vertex subset $U$ is denoted $G\left[U\right]$, the neighbourhood of a vertex $v$ in a graph $G$ is denoted $N_G(v)$, and the degree of $v$ is denoted $\deg_{G}\left(v\right)=|N_G(v)|$.

For a zero-one vector $\boldsymbol{x}\in\left\{ 0,1\right\} ^{n}$,
we write $\left|\boldsymbol{x}\right|$ for the number of entries
that are ones. For a real number $x$, the floor and ceiling functions
are denoted $\floor x=\max\left\{ i\in\ZZ:i\le x\right\} $ and $\ceil x=\min\left\{ i\in\ZZ:i\ge x\right\} $.
Finally, all logarithms are in base $e$.

\subsection{Concentration inequalities}

For the convenience of the reader, in this section we collect some standard concentration inequalities that will be used throughout the paper (since these inequalities are standard, we will refer to them by name and not by their theorem number). First, we will frequently need to use Chernoff bounds for the binomial and hypergeometric distributions. The following bounds  can be found in \cite[Corollary~2.3 and Theorem~2.10]{JLR00}.

\begin{lem}[Chernoff bound]
Suppose $X$ has a binomial or hypergeometric distribution, and consider $0<\varepsilon\le 3/2$. Then
$$\Pr(|X-\E X|\ge \varepsilon \E X)\le 2\exp\left(-\frac{\varepsilon^2}{3}\E X\right).$$
\end{lem}

Second, we will need (a simple consequence of) the Azuma--Hoeffding inequality, as follows. See for example \cite[Corollary~2.27]{JLR00}.

\begin{lem}
Let $X_1,\dots,X_n$ be indepenent random variables, and let $X=f(X_1,\dots,X_n)$ be some function of these random variables. Suppose that if we change the value of some $X_i$, then the value of $X$ changes by at most $c$. Then for every $t>0$, we have
$$\Pr(|X-\E X|\ge t)\le 2\exp\left(-\frac{t^2}{2nc^2}\right).$$
\end{lem}

\section{Generalising Littlewood--Offord to nonnegative polynomials\label{sec:positive-polynomial-LO}}

In this section we prove \cref{thm:positive-polynomial-LO,thm:rough-LO}.
\begin{proof}[Proof of \cref{thm:positive-polynomial-LO}]
Let $H$ be the hypergraph in the definition of $r\left(f\right)$,
with a hyperedge for each coefficient of $f$ with size at least $1$. Let $M$ be
a matching of size $r=r\left(f\right)$ in this hypergraph, and condition
on any outcome of the variables whose indices do not appear in $M$. We may
assume the remaining variables (corresponding to the vertices of $M$) are $\xi_{1},\dots,\xi_{N}$, with
$r\le N\le rd$. Then, $f\left(\x\right)$ is a polynomial in $\xi_{1},\dots,\xi_{N}$.
Abusing notation, we write $\x=\left(\xi_{1},\dots,\xi_{N}\right)$;
we will not need to worry about the specific values of any of the
$\xi_{i}$ we have conditioned on.

Also, it will be more convenient to estimate probabilities of the form $\Pr(|f(\x)-x|<1/2)$ than of the form $\Pr(|f(\x)-x|<1)$. It suffices to show that $\Pr(|f(\x)-x|<1/2)=O(1/\sqrt r)$, because we can cover the length-2 interval $\{y:|f(\x)-x|<1\}$ with three (open) length-1 intervals. For the rest of the proof we fix some $x\in \RR$.

Choose $N_{1}=pN-o\left(r\right)$ and $N_{2}=pN+o\left(r\right)$ so that $\Pr\left(N_{1}\le\left|\x\right|\le N_{2}\right)=1-o\left(1/\sqrt{r}\right)$
(such $N_{1},N_{2}$ exist by the Chernoff bound). Let $\sigma:\left\{ 1,\dots,N\right\} \to\left\{ 1,\dots,N\right\} $
be a uniformly random permutation, and let $\x^{t}$ be the length-$N$
zero-one vector with a 1 in positions $\sigma\left(1\right),\dots,\sigma\left(t\right)$.
Let $Y$ be the number of $t$ satisfying $N_{1}\le t\le N_{2}$ and
$|f\left(\x^{t}\right)-x|<1/2$.

Recall that $\max_{t}\binom{N}{t}p^{t}\left(1-p\right)^{N-t}=\Theta(1/\sqrt N)$ (one can prove this with Stirling's approximation). We can use linearity of expectation to estimate $\E Y$, as follows (recalling that $r=\Theta(n)$).
\begin{align}
\E Y & =\sum_{t=N_{1}}^{N_{2}}\frac{|\{ \boldsymbol{x}\in\left\{ 0,1\right\} ^{N}:\left|\boldsymbol{x}\right|=t,\;|f\left(\boldsymbol{x}\right)-x|<1/2\} |}{\binom{N}{t}}\nonumber \\
 & =\sum_{t=N_{1}}^{N_{2}}\frac{\Pr\left(|f\left(\x\right)-x|<1/2\text{ and }\left|\x\right|=t\right)}{\binom{N}{t}p^{t}\left(1-p\right)^{N-t}}\nonumber \\
 & \ge\frac{1}{\max_{t}\binom{N}{t}p^{t}\left(1-p\right)^{N-t}}\sum_{t=N_{1}}^{N_{2}}\Pr\left(|f\left(\x\right)-x|<1/2\text{ and }\left|\x\right|=t\right)\nonumber \\
 & =\Theta\left(\sqrt{N}\right)\left(\tallphantom\Pr\left(|f\left(\x\right)-x|<1/2\right)-\Pr\left(\left|\x\right|<N_{1}\right)-\Pr\left(\left|\x\right|>N_{2}\right)\right)\nonumber \\
 & =\Theta\left(\sqrt{r}\right)\Pr\left(|f\left(\x\right)-x|<1/2\right)-o\left(1\right).\label{eq:EY-poly}
\end{align}
We can also estimate $\E Y$ a different way, using the relationship between the $\x^t$. Let $X_{t}$ be the number
of $e\in M$ such that $e\cap\sigma\left(\left\{ 1,\dots,t\right\} \right)=\left|e\right|-1$
(that is, all but one of the elements of $e$ have been ``activated''
by time $t$). If $t=pN+o(r)=(1+o(1))pN$ then the probability that any particular $e\in M$ contributes
to $X_{t}$ is $(1+o(1))\left|e\right|p^{\left|e\right|-1}\left(1-p\right)=\Theta\left(1\right)$,
so $\E X_{t}=\Theta\left(N\right)$. Also, changing $\sigma$ by a transposition changes $X_t$ by at most 2, as $M$ is a matching. So, by a McDiarmid-type concentration inequality for
random permutations (see for example \cite[Section~3.2]{McD98}), for each $N_{1}\le t\le N_{2}$ we have
$$\Pr(X_{t}< \E X_t/2)=\exp\left(-\Omega\left(\frac{(\E X_t/2)^2}{N\cdot 2^2}\right)\right)=e^{-\Omega(N)}.$$
Now, observe that $f(\x^t)$ is increasing in $t$, because $f$ has nonnegative coefficients. For $N_{1}\le t\le N_{2}$, let $\mathcal E_{t}$ be the event that $|f\left(\x^t\right)-x|<1/2$, but $f\left(\x^{s}\right)-x\le -1/2$ for $N_1\le s<t$ (that is, $t$ is the first time that $f(\x^t)$ enters the desired range). Note that $Y=0$ unless some $\mathcal E_{t}$ occurs.

For any $t$, condition
on a specific outcome of $\left(\sigma\left(1\right),\dots,\sigma\left(t\right)\right)$
such that $\mathcal E_{t}$ holds and such that $X_{t}\ge \E X_t/2=\Theta(N).$ Let $U$ be the set
of $i\notin\sigma\left(\left\{ 1,\dots,t\right\} \right)$ such that
there is $e\in M$ with $i\in e$ and $e\setminus\left\{ i\right\} \subseteq\sigma\left(\left\{ 1,\dots,t\right\} \right)$. By definition we have $\left|U\right|=X_{t}$. Let $\tau:=\min\{s\in \{t+1,\dots,N\}:\sigma\left(\tau\right)\in U\}$ be the first time that we have
$\sigma\left(\tau\right)\in U$. By the definition of $U$, some edge $e\in M$ will be ``activated'' at time $\tau$, so $f\left(\x^{\tau}\right)\ge f\left(\x^{t}\right)+1$ and in particular $|f\left(\x^{\tau}\right)-x|\ge 1/2$.
Under our conditioning, $\tau-t+1$ is stochastically dominated by the geometric
distribution $\Geom\left(X_{t}/\left(N-t\right)\right)$, which has
expected value $\left(N-t\right)/X_{t}=O\left(1\right)$. We have
proved that $\E\left[Y\cond \mathcal E_{t}\cap\left\{ X_{t}\ge \E X_t/2\right\} \right]=O\left(1\right)$.

Recall that we can have $Y>0$ only if some $\mathcal E_{t}$ occurs, and observe that the $\mathcal E_{t}$ are disjoint and that $Y\le N$ with probability
1. So,
\begin{align*}
\E Y&\le \sum_{t=N_1}^{N_2}\Pr(\mathcal E_t)\E\left[Y\cond \mathcal E_{t}\cap\left\{ X_{t}\ge \E X_t/2\right\} \right]+N\Pr\left(X_{t}<\E X_t/2\text{ for some $t$}\right)\\
&=O\left(1\right)\Pr\left(\mathcal E_{N_1}\cup\dots\cup\mathcal E_{N_2}\right)+N^2e^{-\Omega(N)}=O\left(1\right).
\end{align*}
Combining this with \cref{eq:EY-poly}, the desired result follows.
\end{proof}
\begin{proof}[Proof of \cref{thm:rough-LO}]
We proceed in almost the same way as in the proof of \cref{thm:positive-polynomial-LO}.

Let $\alpha:=\max_{t}\binom{n}{t}p^{t}\left(1-p\right)^{n-t}$, and observe that $\alpha=\Theta(1/\sqrt n)$ (this can be proved with Stirling's approximation; see for example \cite[Proposition~1]{Dun12}). Let $\left|\x\right|$ be the number of ones in $\x$, which has a binomial distribution. By the Chernoff bound, we can choose $n_{1}=pn-O\left(\sqrt{n\log n}\right)$ and $n_{2}=pn+O\left(\sqrt{n\log n}\right)$
such that $\Pr\left(n_{1}\le\left|\x\right|\le n_{2}\right)\ge1-o\left(1/\sqrt{n}\right)$. Observe that $\Pr\left(\left|\x\right|=t\right)\ge n^{-O\left(1\right)}$
for $n_{1}\le t\le n_{2}$ (this can be proved by comparison to the modal probability $\alpha$ or by direct computation using Stirling's inequality; see for example \cite[Proposition~1]{Dun12}). Let $\sigma:\left\{ 1,\dots,n\right\} \to\left\{ 1,\dots,n\right\} $
be a uniformly random permutation, and for each $0\le t\le n$ let $\x^{t}$ be the length-$n$
zero-one vector with a one in positions $\sigma\left(1\right),\dots,\sigma\left(t\right)$, and zeros in the other positions.
Let $Y$ be the number of $t$ satisfying $n_{1}\le t\le n_{2}$ and
$\left|f\left(\x^{t}\right)-x\right|<s$.

The same calculation as in the proof of \cref{thm:positive-polynomial-LO}
gives
\begin{equation}
\E Y\ge \alpha^{-1}\Pr\left(\left|f\left(\x\right)-x\right|<s\right)-o\left(1\right),\label{eq:EY-rough}
\end{equation}
but on the other hand, by the choice of $n_{1},n_{2}$, for all $i$ we have
\begin{align*}
\Pr\left(\Delta_{i}\left(\x^{t}\right)<2s\text{ for some }t\in \{n_{1},\dots,n_{2}\}\right) & \le\sum_{t=n_{1}}^{n_{2}}\Pr\left(\tallphantom\Delta_{i}\left(\x\right)<2s\cond\left|\x\right|=t\right)\\
 & \le\Pr\left(\Delta_{i}\left(\x\right)<2s\right)\sum_{t=n_{1}}^{n_{2}}\frac{1}{\Pr\left(\left|\x\right|=t\right)}\\
 & =n^{-\omega\left(1\right)}O\left(\sqrt{n\log n}\right)n^{O\left(1\right)}=n^{-\omega\left(1\right)}.
\end{align*}
Let $\mathcal E$ be the event that $\Delta_{i}\left(\x^{r}\right)\ge2s$ for all $i$ and all $n_{1}\le r\le n_{2}$, so that $\Pr(\overline{\mathcal E})=n\cdot n^{-\omega(1)}=n^{-\omega(1)}$. Note that $\Delta_i(\x^t)=f(\x^{t+1})-f(\x^{t})$ for $i=\sigma(t+1)$. Therefore, if $\mathcal E$ holds and $\left|f\left(\x^{t}\right)-x\right|<s$ for some $t$, then  $f\left(\x^{r}\right)-x\ge s$ for all $r$ satisfying $t< r\le n_{2}$. That is to say, if $\mathcal E$ holds  then $Y\le1$. Since $Y$ can never be greater than $n_2-n_1+1\le n$, it follows that
\[
\E Y\le 1\cdot \Pr ({\mathcal E})+n\Pr (\overline{\mathcal E})\le 1+n\cdot n^{-\omega\left(1\right)}\le 1+o(1).
\]
Combining this with \cref{eq:EY-rough}, we obtain $\Pr\left(\left|f\left(\x\right)-x\right|<s\right)\le(1+o(1))\alpha=\alpha+o(1/\sqrt n)$, as desired.
\end{proof}

\section{Poisson-type anti-concentration\label{sec:poisson}}

In this section we prove \cref{thm:poisson-L-O}, \cref{prop:weak-bound} and \cref{thm:poisson-nonnegative}.
First, \cref{prop:weak-bound} will be a corollary of the following
non-asymptotic bound for anti-concentration of polynomials of unbiased
coin flips.
\begin{lem}
\label{lem:unbiased-non-asymptotic}Consider a multilinear $n$-variable polynomial
$f$ with degree $d\ge1$, and let $\x\in\Ber\left(1/2\right)^{n}$.
Then for any $\ell\in\RR$,
\[
\Pr\left(f\left(\x\right)=\ell\right)\le1-2^{-d}.
\]
\end{lem}

We prove \cref{lem:unbiased-non-asymptotic} with the combinatorial
Nullstellensatz, whose statement is as follows (see \cite[Theorem~1.2]{Alo99})
\begin{thm}
\label{lem:combinatorial-nullstellensatz}Let $f$ be an $n$-variable polynomial over an arbitrary field $\mathbb F$, with degree $\sum_{i=1}^{n}t_{i}$
(where each $t_{i}$ is a nonnegative integer). Suppose that the coefficient
of $x_{1}^{t_{1}}\dots x_{n}^{t_{n}}$ is nonzero. If $S_{1},\dots,S_{n}$
are subsets of $\mathbb F$ with $\left|S_{i}\right|>t_{i}$, then there is $\boldsymbol{s}\in S_{1}\times\dots\times S_{n}$ with $f(\boldsymbol s)\ne 0$.
\end{thm}

\begin{proof}[Proof of \cref{lem:unbiased-non-asymptotic}]
Suppose without loss of generality that the coefficient of $\xi_{1}\dots\xi_{d}$
is nonzero, and condition on any outcomes for $\xi_{d+1},\dots,\xi_{n}$.
Then, $f\left(\x\right)-\ell$ becomes a degree-$d$ polynomial of $\xi_{1},\dots,\xi_{d}$,
and the coefficient of $\xi_{1}\dots\xi_{d}$ is nonzero. By \cref{lem:combinatorial-nullstellensatz}
with $S_{i}=\left\{ 0,1\right\} $, at least one of the $2^{d}$ equally likely outcomes
of $\left(\xi_{1},\dots,\xi_{d}\right)$ gives $f\left(\x\right)-\ell \ne0$. We have proved that
$$\Pr\left(f(\x)=\ell\cond\xi_{d+1},\dots,\xi_{n}\right)\le 1-2^{-d},$$
and the desired result then follows from the law of total probability.
\end{proof}
Now we prove \cref{prop:weak-bound}.
\begin{proof}[Proof of \cref{prop:weak-bound}]
First note that we can assume $f$ is multilinear, because $\xi_i^2=\xi_i$ for each $i$. Consider any $x$ not equal to the constant coefficient of $f$. Let $\x'\in\Ber\left(2p\right)^n$ and $\g\in\Ber\left(1/2\right)^n$
be independent random vectors, so that $(\xi_{1}'\gamma_{1},\dots,\xi_{n}'\gamma_{n})$ has
the same distribution as $\x$. Note that if we condition on
any outcome of $\x'$ then $f(\xi_{1}'\gamma_{1},\dots,\xi_{n}'\gamma_{n})$ becomes a multilinear polynomial
of $\g$ whose constant coefficient is the same as the constant coefficient
of $f$. If this polynomial is constant then $\Pr\left(f\left(\x\right)=x\right)=0$,
and otherwise \cref{lem:unbiased-non-asymptotic} gives $\Pr\left(f\left(\x\right)=x\right)\le1-2^{-d}$.
\end{proof}
Next, we give a unified proof of \cref{thm:poisson-L-O,thm:poisson-nonnegative}. For $0<p<1$, define
\[
\tau\left(p\right):=\sup_{n\in\NN}\Pr\left(X_{n,p}=1\right)=\sup_{n\in\NN} np(1-p)^{n-1},
\]
where $X_{n,p}\in\Bin\left(n,p\right)$. It is a straightforward computation to determine the limiting behaviour of $\tau(p)$ as $p\to 0$, as follows.

\begin{lem}
We have $\tau\left(p\right)\le 1/e+o_{p\to0}(1)$.
\end{lem}
\begin{proof}
For $0<p<1$, define $\eta_{p}:\left[0,\infty\right)\to\RR$ by $\eta_{p}\left(n\right):=np\left(1-p\right)^{n-1}$.
We compute
\[
\eta_{p}'\left(n\right)=p\left(1-p\right)^{n-1}\left(1+n\log\left(1-p\right)\right),
\]
so $\eta_{p}'\left(n\right)=0$ only when $n=-1/\log\left(1-p\right)$.
Since $\eta_{p}\left(0\right)=0$ and $\eta_{p}\left(n\right)\to0$
as $n\to\infty$, we have 
\[
\tau\left(p\right)\le\sup_{n\in\left[0,\infty\right)}\eta_{p}\left(n\right)=\eta_{p}\left(-1/\log\left(1-p\right)\right)=\frac{-p}{e\left(1-p\right)\log\left(1-p\right)}.
\]
This converges to $1/e$ as $p\to0$, by L'H\^opital's rule.
\end{proof}

The following lemma
then implies \cref{thm:poisson-L-O,thm:poisson-nonnegative}.
\begin{lem}
Let $f$ be an $n$-variable polynomial, with zero constant coefficient,
which is either of degree 1 or has all coefficients nonnegative. Consider
any $p\in\left(0,1\right)$ and $\ell\ne0$. Then
\[
\Pr\left(f\left(\x\right)=\ell\right)\le\tau\left(p\right),
\]
where $\x\in\Ber\left(p\right)^{n}$.
\end{lem}

\begin{proof}
We prove this by induction on $n$ (it is trivially true for $n=0$). So, consider some $n>0$, and assume that the statement is true for the case $n-1$. As in the proof of \cref{lem:unbiased-non-asymptotic} we can assume that $f$ is multilinear.

If $a_{1}=\dots=a_{n}=\ell$ then the only way we can have $f\left(\x\right)=\ell$ is if exactly one of the $\xi_{i}$ is equal to one. So, in this case $\Pr\left(f\left(\x\right)=\ell\right)=\Pr\left(X_{n,p}=1\right)\le\tau\left(p\right)$.

Otherwise, there must be some $a_{j}\ne\ell$. Suppose without
loss of generality that $a_{n}\ne\ell$, so $\ell-a_{n}\ne0$. Define
$\left(n-1\right)$-variable polynomials $g$ and $h$ by $f\left(x_{1},\dots,x_{n}\right)=g\left(x_{1},\dots,x_{n-1}\right)+a_{n}x_{n}+h\left(x_{1},\dots,x_{n-1}\right)x_{n}$, and observe that $g$ and $h$ both have zero constant coefficient. Let $\x'=(\xi_1,\dots,\xi_{n-1})$.
Applying the induction hypothesis to $g$ and $g+h$ gives 
\begin{align*}
\Pr\left(f\left(\x\right)=\ell\right) & =\Pr\left(\xi_{n}=0\right)\Pr\left(f\left(\x\right)=\ell\cond\xi_{n}=0\right)+\Pr\left(\xi_{n}=1\right)\Pr\left(f\left(\x\right)=\ell\cond\xi_{n}=1\right)\\
 & =\left(1-p\right)\Pr\left(g\left(\x'\right)=\ell\right)+p\Pr\left(g\left(\x'\right)+h\left(\x'\right)=\ell-a_{n}\right)\\
 & \le\left(1-p\right)\tau\left(p\right)+p\tau\left(p\right)\\
 & \le\tau\left(p\right).\tag*{\qedhere}
\end{align*}
\end{proof}

\section{Anti-concentration of the edge-statistic\label{sec:edge-ber}}

In this section we give the short proofs of \cref{prop:bernoulli-sqrt,prop:bernoulli-1-e}.
First, note that \cref{prop:bernoulli-1-e} is an immediate consequence
of \cref{thm:poisson-nonnegative}.
\begin{proof}[Proof of \cref{prop:bernoulli-1-e}]
Let $\x\in\Ber\left(k/n\right)^{n}$, and let $f\left(\x\right):=\sum_{e\in E(G)}\prod_{i\in e} \xi_i$. This polynomial has nonnegative coefficients. Then, $X_{G,k}^{\Ber}$
has the same distribution as $f\left(\x\right)$, and since we are
assuming $n/k\to\infty$ we can apply \cref{thm:poisson-nonnegative}.
\end{proof}
Next, it is almost as simple to deduce \cref{prop:bernoulli-sqrt}
from \cref{thm:rough-LO}. We need the following well-known lemma,
which may be proved by induction, iteratively deleting any vertices with degree less than $d/r$ (see for example \cite[Lemma~2.5]{FKSnew}).
\begin{lem}
\label{lem:average-min-degree}Let $G$ be an $r$-uniform hypergraph
with average degree $d$. Then $G$ has an induced subgraph with minimum
degree at least $d/r$.
\end{lem}

\begin{proof}[Proof of \cref{prop:bernoulli-sqrt}]
We can obtain our random subset $A^{\Ber}$ by first sampling a random
subset $A'$, where each element is present with probability $2k/n$,
then deleting each element from $A'$ with probability $1/2$. By
the Chernoff bound, with probability $1-e^{-\Omega\left(k\right)}$
we have $\left(3/2\right)k\le\left|A'\right|\le3k$; consider
such an outcome of $A'$, and let $G':=G\left[A'\right]$. Let $X$ be a random variable having the same distribution as $X$ conditioned on this outcome of $A'$. It now suffices
to show that $\Pr\left(\left|X-\ell\right|\le k^{r-1}\right)=O\left(1/\sqrt{k}\right)$.

Now, note that $\E X=e\left(G'\right)/4$, and observe that
deleting an element of $A^{\Ber}$ can change $X$ by at most
$\left|A'\right|^{r-1}$. If $e\left(G'\right)\le \ell$ then $|\ell-\E X|\ge 3\ell/4\ge \ell/2+k^{r-1}$ and therefore by
the Azuma--Hoeffding inequality we have
\[
\Pr\left(\left|X-\ell\right|\le k^{r-1}\right)\le \Pr\left(\left|X-\E X \right|\ge \ell/2\right)\le\exp\left(-\Omega\left(\frac{\left(\ell/2\right)^{2}}{\left|A'\right|^{2\left(r-1\right)}\cdot\left|A'\right|}\right)\right)=e^{-\Omega\left(k\right)}.
\]
So, we may assume that $e\left(G'\right)\ge\ell=\Omega\left(k^{r}\right)$.
By \cref{lem:average-min-degree}, this implies that $G'$ has an induced
subgraph $G'\left[B\right]$ with minimum degree $\Omega\left(k^{r-1}\right)$.
Condition on any outcome of $A^{\Ber}\setminus B$, and for $i\in B$
let $\xi_{i}:=\one_{i\in A^{\Ber}}$, so that $X_{G,k}$ can be viewed
as a function of $\left(\xi_{i}\right)_{i\in B}$. Recalling the definition of $\Delta_i$ from \cref{thm:rough-LO} we have $\Delta_i=\left|N_G(i)\cap A^{\Ber}\right|\ge \left|N_G(i)\cap A^{\Ber}\cap B\right|$. By our minimum degree assumption, $\left|N_G(i)\cap A^{\Ber}\cap B\right|$ has a binomial distribution with mean $\Omega(k^{r-1})$, so by the Chernoff bound, with probability $1-e^{-\Omega\left(k^{r-1}\right)}$ each
$\Delta_{i}=\Omega\left(k^{r-1}\right)$, which allows us to apply
\cref{thm:rough-LO}. (This gives us a bound for the probability that $X$ falls in an interval of length $\Omega(k^{r-1})$, which suffices because we can cover any interval of length $2k^{r-1}$ with $O(1)$ such intervals).
\end{proof}

\section{Anti-concentration for subgraph counts in random graphs\label{sec:subgraph-counts}}

First we give the simple deduction of \cref{thm:general-random-graphs}
from \cref{thm:rough-LO}.
\begin{proof}[Proof of \cref{thm:general-random-graphs}]
For this proof it is convenient to redefine $X_H$ to count \emph{labelled} copies of $H$ (this changes the anti-concentration behaviour by a constant factor depending on the number of automorphisms of $H$). For a pair of distinct vertices $x,y\in V(G)$, define $\Delta_{x,y}$ to
be the difference
\[
X_{H}\left(G+\left\{ x,y\right\} \right)-X_{H}\left(G-\left\{ x,y\right\} \right)
\]
(that is, the number of copies of $H$ that would be created or destroyed
by flipping the status of $\left\{ x,y\right\} $). Observe that 
\begin{equation}
\E\Delta_{x,y}=2e(H) p^{e(H)-1}n(n-1)\dots(n-h+1)=\Omega\left(n^{h-2}\right),\label{eq:ED}
\end{equation}
Also, observe that for any vertex other than $x$ or $y$, changing the set of edges adjacent to that vertex can affect $\Delta_{x,y}$ by at most $O(n^{h-3})$. So, by the Azuma--Hoeffding inequality (with the vertex exposure martingale) it follows that
\begin{equation}
\Pr(\Delta_{x,y}\le \E\Delta_{x,y}/2)=\exp\left(-\Omega\left(\frac{n^{2(h-2)}}{n\cdot n^{2(h-3)}}\right)\right)=n^{-\omega(1)}.\label{eq:delta-concentration}
\end{equation}
Let $D$ be the common value of the $\E\Delta_{x,y}$. We now apply
\cref{thm:rough-LO} with $\binom n 2=\Theta(n^2)$ variables and with $s=D/4=\Omega(n^{h-2})$, and observe that the interval $\{r\in \RR:|r-x|\le n^{h-2}\}$ (having length $O(n^{h-2})$) can be covered by $O(1)$ intervals of length $2s$.
\end{proof}

Next, we turn to \cref{thm:clique-counts}. First we illustrate the high-level strategy of the proof, which is
by induction on $h$. Recall that $X_{K_{h-1}}=X_{K_{h-1}}(G)$ is the number of copies of $K_{h}$ in $G\in \GG(n,p)$.
Now, let $X$ be the number of copies of $K_{h}$ in $G\in \GG(n,p)$ which contain some fixed vertex $v$ (this
is equal to the number of copies of $K_{h-1}$ in $G\left[N_{G}\left(v\right)\right]$).
Then, we have the decomposition $X_{K_{h}}=X_{K_{h}}\left(G-v\right)+X$,
where $X_{K_{h}}\left(G-v\right)$ is typically much larger than $X$.
One may then hope to establish anti-concentration of $X_{K_{h}}$ by
first showing that $X_{K_{h}}\left(G-v\right)$ is anti-concentrated
at a ``coarse'' scale as in \cref{thm:general-random-graphs}, then
establishing anti-concentration of $X$ on a finer scale. For this
second step, we may first observe that the approximate value of $X$
is primarily driven by $\left|N_{G}\left(v\right)\right|$ (which
has a binomial distribution and is thus easy to study), and if we
condition on $N_{G}\left(v\right)$ then $X$ is the number of copies
of $K_{h-1}$ in a fixed vertex subset of $G-v$, which we may study
with the induction hypothesis.

The main complication with this approach is that it does not suffice
to analyse $X_{K_{h}}\left(G-v\right)$ and $X$ separately, because
in principle they could correlate with each other in a way that increases
the concentration probabilities. So, we must analyse the concentration
behaviour of $X$ conditioned on an outcome of $G-v$. Our approach
is to show that $G-v$ is very likely to have certain properties that
ensure that, conditioned on this outcome of $G-v$, $X$ has approximately the concentration behaviour we would expect unconditionally. For this, we will need something
a bit stronger than \cref{thm:clique-counts} as our induction hypothesis,
as follows.
\begin{defn}For real numbers $c\in (0,1/2)$ and $q\in (0,1)$, we say
an $n$-vertex graph $G$ is \emph{$\left(c,q,h\right)$-dispersed}
if for all $cn\le k\le (1-c)n$ and all $\ell$, the number of induced subgraphs of $G$ with $k$ vertices
and exactly $\ell$ copies of $K_{h}$ is at most $\binom{n}{k}q$.
\end{defn}

\begin{thm}
\label{thm:dispersed}For any constants $c\in (0,1/2)$, $p\in\left(0,1\right)$ and
$h\in\NN$, there are functions $\alpha=\alpha_{h,p,c}$ and $\phi=\phi_{h,p,c}$, with $\lim_{n\to \infty}\alpha(n)= 0$ and $\lim_{n\to \infty}\phi(n)= \infty$, such that the random graph $G\in\GG\left(n,p\right)$ is $\left(c,n^{\alpha+1-h},h\right)$-dispersed
with probability at least $1-n^{-\phi}$.
\end{thm}

To see that \cref{thm:dispersed} implies \cref{thm:clique-counts},
observe that we can obtain a random graph $G\in\GG\left(n,p\right)$
by first taking a random graph $G'\in\GG\left(2n,p\right)$, and
then taking a random subset of $n$ vertices of $G'$. \cref{thm:dispersed} tells us that $G'$ is very likely to be $\left(1/3,n^{o\left(1\right)+1-h},h\right)$)-dispersed, and if it is, then the definition of being dispersed gives the required bound on the point probabilities of $X_{K_h}$.

\begin{proof}[Proof of \cref{thm:dispersed}]As in the proof of \cref{thm:general-random-graphs}, we count labelled cliques, which affects $X_H$ by a constant factor $h!$. The proof is by induction on $h$ (the case $h=1$ is trivial). Fix
$\varepsilon>0$ and $t\in\NN$. We will prove that with probability
at least $1-n^{o\left(1\right)+1+h-\varepsilon t}$, $G\in \GG(n,p)$ is $\left(c,n^{\varepsilon+1-h},h\right)$-dispersed
(asymptotics are allowed to depend on $t$ and $\varepsilon$, which
we view as fixed constants for most of the proof). After we have proved this, we can then let $\varepsilon\to0$
and $t\varepsilon\to\infty$.

For $cn\le r\le (1-c)n$ and $0\le \ell\le \binom{n}{h}$, let $Z_{r,\ell}$ be the number
of sets of $r$ vertices in $G$ that induce exactly $\ell$ copies of $K_{h}$. We need to show that with probability $1-n^{o(1)+1+h-\varepsilon t}$ we have $Z_{r,\ell}\le n^{\varepsilon+1-h}\binom{n}{r}$ for all $cn\le r\le (1-c)n$ and all $0\le \ell \le \binom{n}{h}$.
We upper-bound $\E Z_{r,\ell}^{t}$. Note that if we randomly choose
a sequence $S_{1},\dots,S_{t}$ of $r$-vertex sets (with replacement),
then with probability $1-n^{-\omega\left(1\right)}$ we have
\begin{align}
\left|S_{1}\cup\dots\cup S_{i-1}\right|&\le 1-(c/2)^{i-1}n,\notag{}\\
\left|S_{i}\setminus\left(S_{1}\cup\dots\cup S_{i-1}\right)\right|&\ge (c/2)^{i}n=\Omega(n)\label{eq:exposure}
\end{align}
for each $i\in \{1,\dots ,t\}$. This can be proved by repeatedly applying a Chernoff bound
for the hypergeometric distribution.

Let $\mathcal S$ be the collection of sequences $(S_{1},\dots,S_{t})$ which satisfy \cref{eq:exposure} for each $i\in \{1,\dots,t\}$. For any $(S_{1},\dots,S_{t})\in \mathcal S$, with $G_{i}:=G\left[S_{i}\right]$, we wish to prove that 
\begin{equation}
\Pr\left(X_{K_{h}}\left(G_{i}\right)=\ell\text{ for each }i\right)\le n^{o\left(1\right)+\left(1-h\right)t}.\label{eq:joint-anti-concentration}
\end{equation}
It will follow from \cref{eq:joint-anti-concentration} that
\begin{align*}
\E Z_{r,\ell}^{t}&=\sum_{(S_1,\dots,S_t)\in \mathcal S}\!\!\Pr\left(X_{K_{h}}\left(G[S_i]\right)=\ell\text{ for each }i\right)+\sum_{(S_1,\dots,S_t)\notin \mathcal S}\!\!\Pr\left(X_{K_{h}}\left(G[S_i]\right)=\ell\text{ for each }i\right)\\
 &\le \binom{n}{r}^{t}n^{o\left(1\right)+t\left(1-h\right)}+\binom{n}{r}^{t}n^{-\omega\left(1\right)}\\
&=n^{o\left(1\right)+t\left(1-h\right)}\binom{n}{r}^{t},
\end{align*}
so
\[
\Pr\left(Z_{r,\ell}\ge n^{\varepsilon+1-h}\binom{n}{r}\right)=\Pr\left(Z_{r,\ell}^{t}\ge n^{t\left(\varepsilon+1-h\right)}\binom{n}{r}^{t}\right)\le\frac{\E Z_{r,\ell}^{t}}{n^{t\left(\varepsilon+1-h\right)}\binom{n}{r}^{t}}=n^{o\left(1\right)-t\varepsilon}.
\]
We can then take a union bound over all the (at most
$n\binom{n}{h}\le n^{h+1}$) possibilities for $r,\ell$.

So, it suffices to prove \cref{eq:joint-anti-concentration}. For the rest of the proof we fix a sequence $(S_{1},\dots,S_{t})\in \mathcal S$. The $t$
events $X_{K_{h}}\left(G_{i}\right)=\ell$ are not independent, but
by the choice of $S_{1},\dots,S_{t}$, for each $i$ there is still
a lot of randomness in $G_{i}$ after exposing outcomes of $G_{1},\dots,G_{i-1}$.
The plan is to show that for each $i$, if we condition on an outcome
of $G_{i}^{\cap}:=G\left[S_{i}\cap\left(S_{1}\cup\dots\cup S_{i-1}\right)\right]$,
then unless $G_{i}^{\cap}$ has some atypical properties, there is
still enough randomness to guarantee $\Pr\left(X_{K_{h}}\left(G_{i}\right)=\ell\right)\le n^{o\left(1\right)+\left(1-h\right)}.$

For each $i$ fix some $v_{i}\in S_{i}\setminus\left(S_{1}\cup\dots\cup S_{i-1}\right)$ (which is possible by \cref{eq:exposure}),
let $N_{i}=N_{G}\left(v_{i}\right)\cap S_{i}$, and define
$$X_{i}=X_{K_{h}}\left(G_{i}\right)-X_{K_{h}}\left(G_{i}-v_{i}\right)=X_{K_{h-1}}(G[N_i])$$
to be the number of copies of $K_{h}$ in $G_i$ which
contain $v_{i}$.
Also, let $n'=r-1=\Omega(n)$ be the common size of the sets $S_{i}\setminus\{v_i\}$, let $c'=\min\{p,1-p\}/2$,
let $I=\left\{ k\in\NN:c'n'\le k\le (1-c')n'\right\}$
and let $E_{k}=\E\left[\tallerphantom X_{i}\cond\left|N_i\right|=k\right]=p^{\binom{h-1}{2}}\binom{k}{h-1}$. Let $\beta=\alpha_{h-1,p,c'}(n')=o(1)$ and $\psi=\left(\log n\right)^{1/2}=\omega(1)$, recalling the notation in the statement of \cref{thm:dispersed}.
We say that an outcome of $G_{i}-v_{i}$ is \emph{good} if
\begin{enumerate}
\item it is $\left(c',(n')^{\beta+2-h},h-1\right)$-dispersed;
\item for each $k\in I$, at most $\binom{n'}{k}n^{-\psi}$
size-$k$ subsets $S\subseteq S_{i}\setminus\left\{ v_{i}\right\} $
fail to satisfy
\[
\left|X_{K_{h-1}}\left(G_{i}\left[S\right]\right)-E_{k}\right|\le n^{h-2}\log n.
\]
\end{enumerate}
Then, for $\left\{ x,y\right\} \subseteq S_{i}\setminus\left\{ v_{i}\right\} $,
define $\Delta_{x,y}^{\left(i\right)}$ to be the difference
\[
X_{K_{h}}\left(\left(G_{i}-v_{i}\right)+\left\{ x,y\right\} \right)-X_{K_{h}}\left(\left(G_{i}-v_{i}\right)-\left\{ x,y\right\} \right)
\]
(that is, the number of copies of $K_{h}$ in $G_{i}-v_{i}$ that
would be created or destroyed by flipping the status of $\left\{ x,y\right\} $).
Note that each $\E\Delta_{x,y}^{\left(i\right)}$ is equal to some common value $D=\Theta((n')^{h-2})=\Theta(n^{h-2})$, with essentially the same calculation as in \cref{eq:ED}.
Fix some $\chi=\omega(1)$ that grows sufficiently slowly to satisfy certain inequalities we will encounter later in the proof. Say that an outcome $G^{*}$ of $G_{i}^{\cap}$ is \emph{good-inducing}
if
\[
\Pr\left(G_{i}-v_{i}\text{ is good},\;\Delta_{x,y}^{\left(i\right)}\ge D/2\text{ for all }x,y\in S_{i}\setminus\{v_i\}\cond G_{i}^{\cap}=G^{*}\right)\ge1-n^{-\chi}.
\]
We now break the remainder of the proof into a sequence of claims.
First, we need to show that it is very likely that each $G_{i}^{\cap}$
is good-inducing (here we specify $\chi$).
\begin{claim}
\label{claim:probably-good-inducing}There is $\chi=\omega(1)$ such that $G_{i}^{\cap}$ is good-inducing
for each $i\in \{1,\dots ,t\}$, with probability $1-n^{-\omega\left(1\right)}$.
\end{claim}

We defer the proof of \cref{claim:probably-good-inducing} until later.
It will be a fairly straightforward consequence of the induction hypothesis
and a concentration inequality. Next, recalling \cref{eq:exposure}, note that after exposing $G_i^\cap$ there are still $\Omega(n^2)$ edges of $G_i$ left unexposed. So, we can apply \cref{thm:rough-LO} with $s=D/4$ (and $O(m)$ different values of $x$) to prove the following claim, establishing anti-concentration of $X_{K_{h}}$
at a ``coarse'' scale.
\begin{claim}
\label{claim:G'-loose}For any $x\in\RR$, any real $m\ge 1$ and any good-inducing outcome
$G^{*}$ of $G_{i}^{\cap}$, we have 
\[
\Pr\left(\left|X_{K_{h}}\left(G_{i}-v_{i}\right)-x\right|<m n^{h-2}\cond G_{i}^{\cap}=G^{*}\right)= O\left(\frac{m}{n}\right).
\]
\end{claim}

Then, note that if $G_{i}-v_{i}$ is good, typically
$X_{i}=X_{K_{h-1}}(G[N_i])$ is approximately equal to $E_{\left|N_{i}\right|}$ (specifically, this follows from the second property of being good). So, the following
claim establishes anti-concentration of $X_{i}$.
\begin{claim}
\label{claim:E-loose}For any $x\in\RR$, we have 
$$
\Pr\left(\left|E_{\left|N_{i}\right|}-x\right|\le n^{h-2}\log n\right) = O\left(\frac{\log n}{\sqrt{n}}\right).$$
Further, we have
$$
\Pr\left(\left|E_{\left|N_{i}\right|}-\E X_{i}\right|>n^{h-3/2}\log n\right) =n^{-\omega\left(1\right)}.
$$
\end{claim}

We defer the proof of \cref{claim:E-loose} until later. The proof
is fairly simple, since $\left|N_{i}\right|$ is binomially distributed, and we have an explicit formula for $E_k$. Next, recalling that $X_{i}$ is the number of copies
of $K_{h-1}$ in $G[N_{i}]$, the following claim is a direct consequence
of the first property of being good (that $G_i-v_i$ is $\left(c',(n')^{\beta+2-h},h-1\right)$-dispersed). Indeed, after conditioning on the event that $\left|N_{i}\right|=k$, note that $N_i$ is a uniformly random $k$-vertex subset of $G_i-v_i$.
\begin{claim}
\label{claim:dispersedness-application}For any $x\in\RR$, any $k\in I$,
and any good outcome $G'$ of $G_{i}-v_{i}$, we have 
\[
\Pr\left(X_{i}=x\cond G_{i}-v_{i}=G',\;\left|N_{i}\right|=k\right)\le n^{\beta+2-h}= n^{o\left(1\right)+2-h}.
\]
\end{claim}
Finally, the following claim follows directly from the Chernoff bound, since $|N_i|$ has a binomial distribution with parameters $n'=\Omega(n)$ and $p$.
\begin{claim}\label{claim:chernoff}
For each $i$,
\[
\Pr\left(\left|N_{i}\right|\notin I\right)=n^{-\omega\left(1\right)}
\]
\end{claim}
Before proving \cref{claim:probably-good-inducing,claim:E-loose},
we show how the above claims can be used to deduce \cref{eq:joint-anti-concentration}.
Let $T_{x}$ be the set of all $k\in \NN$ such that $\left|E_{k}-x\right|\le n^{h-2}\log n$. Then, $\Pr\left(\left|N_{i}\right|\in T_{x}\right)=O(\log n/\sqrt n)$ by the first part of \cref{claim:E-loose}. Next, consider any good outcome $G'$ of $G_{i}-v_{i}$. If $|N_i|\notin T_{x}$ then in order to have $X_i=x$ we must have $|X_i-E_{|N_i|}|>n^{h-2}\log n$. So, by the second property of being good, we have 
\[
\Pr\left(X_{i}=x\cond G_{i}-v_{i}=G',\;|N_i|\in I\setminus T_{x}\right)\le n^{-\psi}=n^{-\omega\left(1\right)}.
\]
By \cref{claim:dispersedness-application,claim:chernoff}, it follows that
for any $x\in\RR$ we have
\begin{align}
 & \Pr\left(X_{i}=x\cond G_{i}-v_{i}=G'\right)\nonumber \\
 & \qquad\le\sum_{k\in I\cap T_{x}}\Pr\left(X_{i}=x\cond G_{i}-v_{i}=G',\;|N_i|=k\right)\cdot \Pr(|N_i|=k)\nonumber \\
 & \qquad\qquad+\Pr\left(X_{i}=x\cond G_{i}-v_{i}=G',\;|N_i|\in I\setminus T_{x}\right)+\Pr\left(\left|N_{i}\right|\notin I\right)\nonumber \\
 & \qquad\le \Pr\left(\left|N_{i}\right|\in T_{x}\right)n^{o\left(1\right)+2-h}+n^{-\omega\left(1\right)}+n^{-\omega\left(1\right)}=n^{o\left(1\right)+3/2-h}.\label{eq:fine}
\end{align}
For $i\in \{1,\dots,t\}$, let $\mathcal{F}_{i}$ be the event
that $\left|X_{K_{h}}\left(G_{i}-v_{i}\right)+\E X_{i}-\ell\right|\le 2 n^{h-3/2}\log n$ (note that $\E X_i$ is an unconditional expectation and $\mathcal{F}_i$ only depends on $G_i-v_i$). By \cref{claim:G'-loose} with $x=\ell-\E X_i$ and $m=2\sqrt n \log n$, for any good-inducing outcome $G^{*}$ of $G_{i}^{\cap}$ we have $\Pr\left(\mathcal{F}_{i}\cond G_{i}^{\cap}=G^{*}\right)=O(\log n/\sqrt n)$. Also, by the second part of \cref{claim:E-loose} and the second property of being good, for any good outcome $G'$ of $G_{i}-v_{i}$ (not satisfying $\mathcal F_i$)
we have
\[
\Pr\left(X_{i}=\ell-X_{K_h}\left(G_{i}-v_{i}\right)\cond\overline{\mathcal{F}_{i}},\;G_{i}-v_{i}=G'\right)=n^{-\omega\left(1\right)}.
\]
Using \cref{eq:fine}, for any good-inducing
outcome $G^{*}$ of $G_{i}^{\cap}$ we then have
\begin{align}
 & \Pr\left(X_{K_{h}}\left(G_{i}\right)=\ell\cond G_{i}^{\cap}=G^{*}\right)\notag{}\\
 & \qquad=\Pr\left(X_{K_{h}}(G_i-v_i)+X_i=\ell\cond G_{i}^{\cap}=G^{*}\right)\notag{}\\
 & \qquad\le \Pr\left(\mathcal{F}_{i}\cond G_{i}^{\cap}=G^{*}\right)\Pr\left(X_{i}=\ell-X_{K_h}\left(G_{i}-v_{i}\right)\cond\mathcal{F}_{i},\;G_{i}-v_{i}\text{ is good},\;G_{i}^{\cap}=G^{*}\right)\notag{}\\
 & \qquad\qquad+\Pr\left(X_{i}=\ell-X_{K_h}\left(G_{i}-v_{i}\right)\cond\overline{\mathcal{F}_{i}},\;G_{i}-v_{i}\text{ is good},\;G_{i}^{\cap}=G^{*}\right)\notag{}\\
 &\qquad\qquad+\Pr\left(G_{i}-v_{i}\text{ is not good}\cond G_{i}^{\cap}=G^{*}\right)\notag{}\\
 & \qquad\le O\left(\frac{\log n}{\sqrt{n}}\right) n^{o\left(1\right)+3/2-h}+n^{-\omega\left(1\right)}+n^{-\omega\left(1\right)}=n^{o\left(1\right)+1-h}\label{eq:anticoncentration-conclusion}.
\end{align}
Now, let $\mathcal{H}_{i}$ be the event that $X_{K_{h}}\left(G_{i}\right)=\ell$
and that $G_{i+1}^{\cap}$ is good-inducing (if $i=t$ this is just
the event that $X_{K_{h}}\left(G_{t}\right)=\ell$). Observe that $G_1^\cap=\emptyset$ is not actually random, so \cref{claim:probably-good-inducing} implies that it is always good-inducing. Applying \cref{eq:anticoncentration-conclusion} we have
\begin{align*}
\Pr\left(\mathcal{H}_{i}\cond\mathcal{H}_{1},\dots,\mathcal{H}_{i-1}\right)\le\Pr\left(X_{K_{h}}\left(G_{i}\right)=\ell\cond\mathcal{H}_{1},\dots,\mathcal{H}_{i-1}\right)
\le n^{o\left(1\right)+1-h},
\end{align*}
so
\[
\Pr\left(\mathcal{H}_{1}\cap\dots\cap\mathcal{H}_{t}\right)=\prod_{i=1}^{t}\Pr\left(\mathcal{H}_{i}\cond\mathcal{H}_{1},\dots,\mathcal{H}_{i-1}\right)\le n^{o\left(1\right)+\left(1-h\right)t}.
\]
Finally, by \cref{claim:probably-good-inducing}, we have
\begin{align*}
\Pr\left(X_{K_{h}}\left(G_{i}\right)=\ell\text{ for each }i\right) & \le\Pr\left(\mathcal{H}_{1}\cap\dots\cap\mathcal{H}_{t}\right)+\Pr\left(\text{some }G_{i}^{\cap}\text{ is not good-inducing}\right)\\
 & \le n^{o\left(1\right)+\left(1-h\right)t}+n^{-\omega\left(1\right)}=n^{o\left(1\right)+\left(1-h\right)t},
\end{align*}
concluding the proof of \cref{eq:joint-anti-concentration}. 
Pending the proofs of \cref{claim:probably-good-inducing,claim:E-loose}, which are given below, this completes the proof of \cref{thm:dispersed}. 
\end{proof}

\cref{claim:probably-good-inducing} will be a consequence of the law of total expectation and the following claim.
\begin{claim}\label{claim:probably-good}
Let $\mathcal{A}_{i}$ be the event that $G_{i}-v_{i}$ fails to
satisfy the first property of being good, let $\mathcal{B}_{i}$ be the event that it fails to satisfy the second property of being good, and let $\mathcal{C}_{i}$ be the event that $\Delta_{x,y}^{\left(i\right)}<D/2=\E\Delta_{x,y}^{\left(i\right)}/2$ for some $x,y\in S_{i}\setminus\{v_i\}$. Then, for each $i\in \{1,\dots,t\}$, we have $\Pr\left(\mathcal{A}_{i}\cup\mathcal{B}_{i}\cup\mathcal{C}_{i}\right)= n^{-\omega(1)}$.
\end{claim}
\begin{proof}
First, we have $\Pr\left(\mathcal{A}_{i}\right)\le (n')^{-\phi_{h-1,p,c'}(n')}=n^{-\omega(1)}$, by \cref{thm:dispersed} for $h-1$ (which we are assuming as our induction hypothesis).

Second, we have $\Pr\left(\mathcal{C}_{i}\right)=e^{-\Omega(n')}=n^{-\omega(1)}$ with exactly the same argument as in \cref{eq:delta-concentration} in the proof of \cref{thm:general-random-graphs} (using the Azuma--Hoeffding inequality) and the union bound.

Third, we consider $\mathcal{B}_i$. For each $k\in I$
and each subset $S\subseteq S_{i}-v_{i}$ of size $k$, consider the random variable $X_{K_{h-1}}\left(G_{i}\left[S\right]\right)$. This random variable has mean $E_k=\Omega(n^{h-1})$ and flipping the status of an edge causes a change of at most $O(n^{h-3})$. So, by the Azuma--Hoeffding
inequality we have
\[
\Pr\left(\left|X_{K_{h-1}}\left(G_{i}\left[S\right]\right)-E_{k}\right|>n^{h-2}\log n\right)=\exp\left(-\Omega\left(\frac{\left(n^{h-2}\log n\right)^2}{n^2\cdot n^{2(h-3)}}\right)\right)=e^{-\Omega\left((\log n)^2\right)}.
\]
Hence, the expected number of subsets $S$ for which $\left|X_{K_{h-1}}\left(G_{i}\left[S\right]\right)-E_{k}\right|>n^{h-2}\log n$
is $\binom{n'}{k}e^{-\Omega\left((\log n)^2\right)}$, and by Markov's
inequality, the probability that this occurs for more than $\binom{n'}{k}e^{-\left(\log n\right)^{3/2}}=\binom{n'}{k}n^{-\psi}$
subsets is at most $e^{-\Omega\left((\log n)^{2}\right)}=n^{-\omega\left(1\right)}$.
We can then take the union bound over all $k\in I\subseteq \{1,\dots,n\}$ to obtain $\Pr\left(\mathcal{B}_{i}\right)\le n^{-\omega(1)}$.
\end{proof}
Now we prove \cref{claim:probably-good-inducing}.
\begin{proof}[Proof of \cref{claim:probably-good-inducing}]
Fix some $i$; we will show that $G_i^\cap$ is good-inducing with probability $n^{-\omega(1)}$. We can then take the union bound over all $i$.

Let $W$ be the random variable $\Pr\left(\mathcal{A}_{i}\cup\mathcal{B}_{i}\cup\mathcal{C}_{i}\cond G_{i}^{\cap}\right)$
(which depends on $G_{i}^{\cap}$). By the law of total expectation and \cref{claim:probably-good}, we have $\E W=\Pr\left(\mathcal{A}_{i}\cup\mathcal{B}_{i}\cup\mathcal{C}_{i}\right)\le f$
for some $f=n^{-\omega\left(1\right)}$, so $\Pr\left(W\ge\sqrt{f}\right)\le\sqrt{f}$
by Markov's inequality. Since $\sqrt{f}$ is still of the form $n^{-\omega\left(1\right)}$, letting $\chi=-\log \sqrt f/\log n$, 
the desired result follows.
\end{proof}
\begin{proof}[Proof of \cref{claim:E-loose}]
For any $k\le n'$, note that
\[
E_{k}-E_{k-1}=p^{\binom{h-1}{2}}\binom{k}{h-1}-p^{\binom{h-1}{2}}\binom{k-1}{h-1}=O\left(n^{h-2}\right).
\]
Now, the second inequality then follows from the Azuma--Hoeffding inequality: we have just observed that adding or removing a vertex from $N_i$ changes $E_{|N_i|}$ by $O\left(n^{h-2}\right)$, so
\[
\Pr\left(\left|E_{\left|N_{i}\right|}-\E X_{i}\right|>n^{h-3/2}\log n\right) =\exp\left(-\Omega\left(\frac{\left(n^{h-3/2}\log n\right)^2}{n\cdot n^{2(h-2)}}\right)\right)=e^{-\Omega((\log n)^2)}=n^{-\omega\left(1\right)}.\]
For the first inequality, we can now assume that $x\ge \E X_i-n^{h-3/2}\log n-n^{h-2}\log n=\Omega(n^{h-1})$. Therefore, we can only have $\left|E_{k}-x\right|\le n^{h-2}\log n$ if $k=\Omega(n)$. If $k=\Omega(n)$ then we can compute $E_{k}-E_{k-1}=\Omega(n^{h-2})$, so there are only $O\left(\log n\right)$ values of $k$ which yield
$\left|E_{k}-x\right|\le n^{h-2}\log n$. Since $\left|N_{i}\right|$ has the binomial distribution $\Bin\left(n',p\right)$, the probability that $\left|N_{i}\right|$
takes one of these values is $O\left(\log n/\sqrt{n}\right)$, recalling that $n'=\Omega(n)$.
\end{proof}

\section{Concluding remarks\label{sec:concluding}}

In this paper we have proved several new anti-concentration inequalities
and given some applications. There are many interesting
directions of future research.

First, we still do not have a complete
understanding of anti-concentration for bounded-degree polynomials
in the ``Gaussian'' regime where $p$ is fixed. Most obviously, it would be
very interesting to remove the polylogarithmic factor from the Meka--Nguyen--Vu inequality, for polynomials which have both positive and negative coefficients. As noted in \cite{KST}, this would imply \cref{conj:sqrt}.

Also, while the Meka--Nguyen--Vu inequality gives an almost-optimal bound on $Q_{f(\x)}(1)$ (for a bounded-degree polynomial $f$ and $\x\in \Ber(p)^n$), our understanding of the whole concentration function $Q_{f(\x)}$ is still quite limited, even for ``dense'' polynomials with many large coefficients. For example, if $f$ has degree $d=O(1)$ and $\Omega(n^d)$ coefficients with absolute value at least 1, then the Meka--Nguyen--Vu inequality gives $Q_{f(\x)}(r)=O((\log n)^{O(1)}(r+1)/\sqrt n)$, whereas it seems likely that the correct bound should be $Q_{f(\x)}(r)=O((r^{1/d}+1)/\sqrt n)$ (attained by the polynomial $\left(x_1+\dots+x_{n}-pn\right)^d$).

Second, regarding the ``Poisson'' regime where $p$ may be a vanishing
function of $n$, we were not able to find a polynomial anti-concentration
inequality that implies \cref{conj:1/e}, which was our initial motivation
for this study. There are invariance principles (see \cite{FKMW16,FM16})
which allow us to compare $X_{G,k}$ to polynomials of Bernoulli random
variables, but as we observed in \cref{subsec:poisson}, in general
there are bounded-degree polynomials yielding point probabilities
much larger than $1/e$. The invariance principles in \cite{FKMW16,FM16}
yield polynomials with special structure (\emph{harmonic} polynomials),
and perhaps it would be feasible to prove an analogue of \cref{thm:poisson-nonnegative}
for such polynomials, which might yield a new proof of \cref{conj:1/e}
and a generalisation for hypergraphs.

So far, in order to avoid trivialities, when we allow $p$ to decrease with
$n$ we have been considering probabilities of the form $\Pr\left(X=x\right)$
only when $x\ne0$. A different way to avoid trivialities would be
to impose that the $a_{i}$ are nonzero and explicitly specify the
dependence of $p$ on $n$; of particular interest may be the Poisson
regime where $p=\lambda/n$ for some constant $\lambda$. However,
in this setting there do not seem to be theorems that are quite as
elegant as \cref{thm:poisson-L-O}. In the case where all the $a_{i}$
are positive, one can imitate Erd\H os' proof of the Erd\H os--Littlewood--Offord
theorem to prove a bound of the form
\[
\Pr\left(a_1\xi_1+\dots+a_n\xi_n=x\right)\le\max_{x\in\ZZ}\Pr\left(Z=x\right)+o\left(1\right),
\]
where $Z$ has the Poisson distribution $\Po\left(\lambda\right)$. If we do not require that the $a_{i}$ are all positive, we get some
more complicated behaviour. We can resolve the linear case with Fourier
analysis, as follows.
\begin{prop}
\label{thm:bessel}Fix $\lambda>0$ and consider a linear polynomial
$X=\sum_{i=1}^{n}a_{i}\xi_{i}$, where $\x\in\Ber\left(\lambda/n\right)^{n}$.
Then for any $x\in\RR$, 
\[
\Pr\left(X=x\right)\le\frac{I_{0}\left(\lambda\right)}{e^{\lambda}}+o\left(1\right),
\]
where
\[
I_{0}\left(\lambda\right)=\sum_{i=0}^{\infty}\frac{\left(\lambda/2\right)^{2i}}{\left(i!\right)^{2}}
\]
is an evaluation of a modified Bessel function of the first kind.
\end{prop}

This bound is best-possible, as can be proved by considering the case where
$a_{1},\dots,a_{\floor{n/2}}=1$ and $a_{\floor{n/2}+1},\dots,a_{n}=-1$.
\begin{proof}[Proof of \cref{thm:bessel}]
First, with a standard reduction we may assume all the $a_i$ are integers. Indeed, we can view $\RR$ as a vector space over the rational numbers $\QQ$, and choose a projection map $P:\RR \to \QQ$ such that $P(a_i)\ne 0$ for each $i$. Clearing denominators by multiplying by some integer $d$, we obtain nonzero integers $a_i'=d P(a_i)$ such that whenever we have $a_1\xi_1+\dots+a_n\xi_n=x$, we have $a_1'\xi_1+\dots+a_n'\xi_n=d P(x)$. That is, any anti-concentration bound for the random variable $a_1'\xi_1+\dots+a_n'\xi_n$ implies the same bound for $a_1\xi_1+\dots+a_n\xi_n$.

So, we assume each $a_i$ is an integer (therefore we may also assume $x$ is an integer). We do Fourier analysis over $\ZZ/N\ZZ$, for some prime $N$ very large relative to $n,x$ and the $a_i$. For all $a\in \ZZ/N\ZZ$, let $f_{a}=\left(1-\lambda/n\right)\delta_{0}+\left(\lambda/n\right)\delta_{a}$,
so that
\[
\hat{f}_{a}\left(k\right)=\frac{\lambda}{n}e^{-2\pi iak/N}+\left(1-\frac{\lambda}{n}\right).
\]
Also, we have
\begin{align*}
\Pr\left(X=x\right) & =f_{a_{1}}*\dots*f_{a_{n}}\left(x\right)\\
 & =\frac{1}{N}\sum_{k=0}^{N-1}e^{2\pi i x k/N}\prod_{j=1}^{n}\hat{f}_{a_{j}}\left(k\right)\\
 & \le\frac{1}{N}\sum_{k=0}^{N-1}\prod_{j=1}^{n}\left|\hat{f}_{a_{j}}\left(k\right)\right|\\
 & \le\prod_{j=1}^{n}\left(\frac{1}{N}\sum_{k=0}^{N-1}\left|\hat{f}_{a_{j}}\left(k\right)\right|^{n}\right)^{1/n}\\
 & =\frac{1}{N}\sum_{k=0}^{N-1}\left|\hat{f}_{1}\left(k\right)\right|^{n}, 
\end{align*}
where the second equality is by the Fourier inversion formula, and the second inequality is by H\"older's inequality. Now, taking $N\to\infty$ gives
$$\Pr(X=x) \le \int_{0}^{1}\left|\frac{\lambda}{n}e^{-2\pi ix}+\left(1-\frac{\lambda}{n}\right)\right|^{n}\d x.$$
As $n\to\infty$ we can compute 
\[
\left|\frac{\lambda}{n}e^{-2\pi ix}+\left(1-\frac{\lambda}{n}\right)\right|^{n}\to\left|e^{-\lambda(1-\cos(2\pi x)+i\sin(2\pi x))}\right|= e^{-\lambda+\lambda\cos\left(2\pi x\right)},
\]
and it is known (see for example \cite[Eq.~(3), p.~181]{Wat95}) that
\[
\int_{0}^{1}e^{\lambda\cos\left(2\pi x\right)}\d x=I_{0}\left(\lambda\right),
\]
so by the dominated convergence theorem, $\Pr\left(X=x\right)\le I_{0}\left(\lambda\right)e^{-\lambda}+o\left(1\right)$.
\end{proof}

We also think it might be interesting to study the situation for general $p$ (in particular, the intermediate
regime between $p=\lambda/n$ ``Poisson'' behaviour and $p=1/2$ ``Gaussian''
behaviour). The linear case would be a good start, as follows.
%In the linear case, we propose the following conjecture.
\begin{question}
\label{conj:general-LO}
Let $\boldsymbol{a}=(a_1,\ldots,a_n)\in\left(\RR\setminus\left\{ 0\right\} \right)^{n}$, $\x\in\Ber\left(p\right)^{n}$ for some $0<p\le1/2$ and $X=a_{1}\xi_{1}+\dots+a_{n}\xi_{n}$. What upper bounds (in terms of $n$ and $p$) can we give on the maximum point probability $Q_X(0)=\max_{x\in \RR}\Pr\left(X=x\right)$?
\end{question}

Also, we remark that the constant $1/e$ in \cref{thm:poisson-L-O,thm:poisson-nonnegative} appears in several other combinatorial and probabilistic problems, such as in a well-known conjecture of Feige~\cite{Fei06}. 

Finally, on the subject of subgraph counts in random graphs, it may also be interesting to study anti-concentration of the number of \emph{induced} copies
$X_{H}'$ of a subgraph $H$ in a random graph $\GG\left(n,p\right)$. (This question was also raised by Meka, Nguyen and Vu~\cite{MNV16}). Using \cref{thm:rough-LO} in the same way as the proof of \cref{thm:general-random-graphs}, one can prove that $\Pr\left(\left|X_{H}'-x\right|\le n^{h-2}\right)=O\left(\frac{1}{n}\right)$, provided $p$ is different to the edge-density of $H$. The natural analogue of \cref{conj:H-count} is that for a fixed graph $H$ and fixed $p\in\left(0,1\right)$, we have
\[
\max_{x\in \NN}\Pr\left(X_{H}'=x\right)=O\left(1/\sqrt{\Var\left(X_{H}'\right)}\right).
\]
We remark that the behaviour of $\sqrt{\Var\left(X_{H}'\right)}$ is not entirely trivial: for most values of $p$ it has order $\Theta(n^{h-1})$, but when $p$ is exactly equal to the edge-density of $H$ it may have order $\Theta(n^{h-3/2})$ or $\Theta(n^{h-2})$ (see \cite[Theorem~6.42]{JLR00}).
%\bibliographystyle{amsplain_initials_nobysame}
%\bibliography{references}

\begin{thebibliography}{10}

\bibitem{Alo99}
N.~Alon, \emph{Combinatorial {N}ullstellensatz}, Combin. Probab. Comput.
  \textbf{8} (1999), no.~1-2, 7--29, Recent trends in combinatorics
  (M\'{a}trah\'{a}za, 1995).

\bibitem{AHKT}
N.~Alon, D.~Hefetz, M.~Krivelevich, and M.~Tyomkyn, \emph{Edge-statistics on
  large graphs}, Combin. Probab. Comput. \textbf{29} (2020), no.~2, 163--189.

\bibitem{BHLP16}
J.~Balogh, P.~Hu, B.~Lidick\'y, and F.~Pfender, \emph{Maximum density of
  induced 5-cycle is achieved by an iterated blow-up of 5-cycle}, European J.
  Combin. \textbf{52} (2016), part~A, 47--58.

\bibitem{BKR89}
A.~D. Barbour, M.~Karo\'{n}ski, and A.~Ruci\'{n}ski, \emph{A central limit
  theorem for decomposable random variables with applications to random
  graphs}, J. Combin. Theory Ser. B \textbf{47} (1989), no.~2, 125--145.

\bibitem{Ber16}
R.~Berkowitz, \emph{A quantitative local limit theorem for triangles in random
  graphs}, arXiv preprint arXiv:1610.01281 (2016).

\bibitem{Ber18}
R.~Berkowitz, \emph{A local limit theorem for cliques in {$G(n,p)$}}, arXiv
  preprint arXiv:1811.03527 (2018).

\bibitem{Bol01}
B.~Bollob\'{a}s, \emph{Random graphs}, second ed., Cambridge Studies in
  Advanced Mathematics, vol.~73, Cambridge University Press, Cambridge, 2001.

\bibitem{BPR00}
B.~Bollob\'{a}s, L.~Pebody, and O.~Riordan, \emph{Contraction-deletion
  invariants for graphs}, J. Combin. Theory Ser. B \textbf{80} (2000), no.~2,
  320--345.

\bibitem{BVW10}
J.~Bourgain, V.~H. Vu, and P.~M. Wood, \emph{On the singularity probability of
  discrete random matrices}, J. Funct. Anal. \textbf{258} (2010), no.~2,
  559--603.

\bibitem{CTV06}
K.~P. Costello, T.~Tao, and V.~Vu, \emph{Random symmetric matrices are almost
  surely nonsingular}, Duke Math. J. \textbf{135} (2006), no.~2, 395--413.

\bibitem{MN04}
A.~de~Mier and M.~Noy, \emph{On graphs determined by their {T}utte
  polynomials}, Graphs Combin. \textbf{20} (2004), no.~1, 105--119.

\bibitem{Dun12}
S.~R. Dunbar, \emph{Topics in probability theory and stochastic processes: The
  moderate deviations result}, 2012, URL:
  \url{https://www.math.unl.edu/~sdunbar1/ProbabilityTheory/Lessons/BernoulliTrials/ModerateDeviations/moderatedeviations.pdf}.
  Last visited on 2018/12/03.

\bibitem{Erd45}
P.~Erd{\H{o}}s, \emph{On a lemma of {L}ittlewood and {O}fford}, Bull. Amer.
  Math. Soc. \textbf{51} (1945), 898--902.

\bibitem{Fei06}
U.~Feige, \emph{On sums of independent random variables with unbounded variance
  and estimating the average degree in a graph}, SIAM J. Comput. \textbf{35}
  (2006), no.~4, 964--984.

\bibitem{FKMW16}
Y.~Filmus, G.~Kindler, E.~Mossel, and K.~Wimmer, \emph{Invariance principle on
  the slice}, 31st {C}onference on {C}omputational {C}omplexity, LIPIcs.
  Leibniz Int. Proc. Inform., vol.~50, Schloss Dagstuhl. Leibniz-Zent. Inform.,
  Wadern, 2016, Art. No. 15, 10 pages.

\bibitem{FM16}
Y.~Filmus and E.~Mossel, \emph{Harmonicity and invariance on slices of the
  {B}oolean cube}, 31st {C}onference on {C}omputational {C}omplexity, LIPIcs.
  Leibniz Int. Proc. Inform., vol.~50, Schloss Dagstuhl. Leibniz-Zent. Inform.,
  Wadern, 2016, Art. No. 16, 13 pages.

\bibitem{FKS2}
J.~Fox, M.~Kwan, and L.~Sauermann, \emph{Anticoncentration for subgraph counts
  in random graphs}, arXiv preprint arXiv:1905.12749 (2019).

\bibitem{FKSnew}
J.~Fox, M.~Kwan, and B.~Sudakov, \emph{Acyclic subgraphs of tournaments with high chromatic number}, arXiv preprint arXiv:1912.07722 (2019).

\bibitem{FS}
J.~Fox and L.~Sauermann, \emph{A completion of the proof of the edge-statistics
  conjecture}, Advances in Combinatorics 2020:4.

\bibitem{GK16}
J.~Gilmer and S.~Kopparty, \emph{A local central limit theorem for triangles in
  a random graph}, Random Structures Algorithms \textbf{48} (2016), no.~4,
  732--750.

\bibitem{HT}
D.~Hefetz and M.~Tyomkyn, \emph{On the inducibility of cycles}, J.
  Combin. Theory Ser. B \textbf{133} (2018), 243--258.

\bibitem{JLR00}
S.~Janson, T.~{\L{}}uczak, and A.~Ruci{\'n}ski, \emph{Random graphs}, Cambridge
  University Press, 2000.

\bibitem{Kal}
O.~Kallenberg, \emph{Foundations of modern probability}, Probability and its
  Applications (New York), Springer-Verlag, New York, 1997.

\bibitem{KV00}
J.~H. Kim and V.~H. Vu, \emph{Concentration of multivariate polynomials and its
  applications}, Combinatorica \textbf{20} (2000), no.~3, 417--434.

\bibitem{KNV}
D.~Kr\'al', S.~Norin, and J.~Volec, \emph{On the exact maximum induced density
  of almost all graphs and their inducibility},  J. Combin. Theory Ser. A \textbf{161} (2019), 359--363.

\bibitem{KST}
M.~Kwan, B.~Sudakov, and T.~Tran, \emph{Anticoncentration for subgraph
  statistics}, J. London Math. Soc. \textbf{99}, no.~3 (2019), 757--777.

\bibitem{LO43}
J.~E. Littlewood and A.~C. Offord, \emph{On the number of real roots of a
  random algebraic equation. {III}}, Rec. Math. [Mat. Sbornik] N.S.
  \textbf{12(54)} (1943), 277--286.

\bibitem{LMP04}
M.~Loebl, J.~Matou\v{s}ek, and O.~Pangr\'{a}c, \emph{Triangles in random
  graphs}, Discrete Math. \textbf{289} (2004), no.~1-3, 181--185.

\bibitem{MMNT}
A.~Martinsson, F.~Mousset, A.~Noever, and M.~Truji{\'c}, \emph{The
  edge-statistics conjecture for $\ell \ll k^{6/5}$}, Israel J. Math. \textbf{234} (2019), no.~2, 677--690.

\bibitem{McD98}
C.~McDiarmid, \emph{Concentration}, Probabilistic methods for algorithmic
  discrete mathematics, Algorithms Combin., vol.~16, Springer, Berlin, 1998,
  pp.~195--248.

\bibitem{MNV16}
R.~Meka, O.~Nguyen, and V.~Vu, \emph{Anti-concentration for polynomials of
  independent random variables}, Theory Comput. \textbf{12} (2016), Paper No.
  11, 16 pages.

\bibitem{NV13}
H.~H. Nguyen and V.~H. Vu, \emph{Small ball probability, inverse theorems, and
  applications}, Erd{\H{o}}s centennial, Bolyai Soc. Math. Stud., vol.~25,
  J\'{a}nos Bolyai Math. Soc., Budapest, 2013, pp.~409--463.

\bibitem{PG75}
N.~Pippenger and M.~C. Golumbic, \emph{The inducibility of graphs}, J.
  Combin. Theory Ser. B \textbf{19} (1975), no.~3, 189--203.

\bibitem{RS96}
J.~Rosi\'nski and G.~Samorodnitsky, \emph{Symmetrization and concentration inequalities for multilinear forms with applications to zero-one laws for {L}\'{e}vy chaos}, Ann. Probab.
  \textbf{24} (1996), no.~1, 422--437.

\bibitem{RV13}
A.~Razborov and E.~Viola, \emph{Real advantage}, ACM Trans. Comput. Theory
  \textbf{5} (2013), no.~4, Art. 17, 8 pages.

\bibitem{Ros11}
N.~Ross, \emph{Fundamentals of {S}tein's method}, Probab. Surv. \textbf{8}
  (2011), 210--293.

\bibitem{TV09a}
T.~Tao and V.~Vu, \emph{From the {L}ittlewood-{O}fford problem to the circular
  law: universality of the spectral distribution of random matrices}, Bull.
  Amer. Math. Soc. (N.S.) \textbf{46} (2009), no.~3, 377--396.

\bibitem{TV09b}
T.~Tao and V.~H. Vu, \emph{Inverse {L}ittlewood-{O}fford theorems and the
  condition number of random discrete matrices}, Ann. of Math. (2) \textbf{169}
  (2009), no.~2, 595--632.

\bibitem{Vu17}
V.~Vu, \emph{Anti-concentration inequalities for polynomials}, A journey
  through discrete mathematics, Springer, Cham, 2017, pp.~801--810.

\bibitem{Wat95}
G.~N. Watson, \emph{A treatise on the theory of {B}essel functions}, Cambridge
  Mathematical Library, Cambridge University Press, Cambridge, 1995, Reprint of
  the second (1944) edition.

\bibitem{Yus}
R.~Yuster, \emph{On the exact maximum induced density of almost all graphs and
  their inducibility}, J.
  Combin. Theory Ser. B \textbf{136} (2019), 81--109.

\end{thebibliography}

\textbf{Acknowledgements.} We would like to thank Van Vu for some clarifications regarding his work with Meka and Nguyen~\cite{MNV16}. Also, after we completed a draft of this paper, Anders Martinsson and Frank Mousset told us that one can give an alternative proof of \cref{thm:poisson-nonnegative} using the ideas in \cite[Section~2]{MMNT}. Finally, in a previous version of this paper we stated a specific conjecture in the setting of \cref{conj:general-LO}, which was observed by Mihir Singhal to be incorrect.

\end{document}